\documentclass[12pt]{article}

\usepackage[normalem]{ulem}

\usepackage{BOONDOX-cal}
\usepackage{amsmath,amsthm}
\usepackage{amsfonts,amsmath}
\usepackage{amssymb}
\usepackage{fullpage}
\usepackage{epsf}
\usepackage{color}
\usepackage{graphicx}
\usepackage{pdflscape}
\usepackage{rotating}

\usepackage{yfonts}

\usepackage{thmtools}
\usepackage{thm-restate}

\usepackage{hyperref}
\usepackage{cleveref}

\usepackage{caption}
\usepackage{subcaption}
\usepackage{mathrsfs}
\usepackage{pgf,tikz}
\usepackage{multirow}

\usepackage{tikzsymbols}

\usetikzlibrary{calc}

\usepackage{mathtools}
\usepackage{xfrac}
\usepackage{relsize}

\usepackage{array}
\usepackage{tabularray}

\newtheorem{theorem}{Theorem}[section]
\newtheorem{lemma}[theorem]{Lemma}

\theoremstyle{definition}

\newtheorem{remark}[theorem]{Remark}

\def\Z{\mathbb{Z}}

\def \cA {{\cal A}}
\def \cB {{\cal B}}
\def \cC {{\cal C}}
\def \cD {{\cal D}}

\def \cX {{\cal X}}

\def \Z {\mathbb Z}

\newcommand{\TPT}[1]{{\mathrm{TPT}}_#1}

\def\Z{\mathbb{Z}}
\def\L{L}
\def\fd{\textgoth{F}}
\def\df{\overrightarrow{\cD}}
\def\db{\overleftarrow{\cD}}
\def\BG{\mathrm{BG}}
\def\MG{\mathrm{MG}}

\def\D{{\cal D}}

\def\a{\alpha}
\def\b{\beta}
\def\g{\gamma}
\def\G{\Gamma}
\def\D{\Delta}
\def\E{\mathcal{E}}
\def\EA{\textgoth{A}}

\definecolor{darkgreen}{RGB}{34, 150, 65}

\definecolor{gold(metallic)}{rgb}{0.83, 0.69, 0.22}
\newcommand{\comment}[1]{}

\newcommand{\Keywords}[1]{\par\noindent{\bfseries Keywords}: #1}
\newcommand{\MSC}[1]{\def\MSCok{\relax}\par\noindent{\bfseries MSC}: #1}

\def\l{\ell}

\def \cC {{\cal C}}

\title{Uniquely 2-colourable 4-cycle decompositions}
\author{\textit{Andrea C.~Burgess\thanks{Department of Mathematics and Statistics, University of New Brunswick, Saint John, NB, E2L~4L5, Canada. \href{mailto:andrea.burgess@unb.ca}{{\fontfamily{cmtt}\selectfont andrea.burgess@unb.ca}}}} \and \textit{David A.~Pike\thanks{Department of Mathematics and Statistics, Memorial University of Newfoundland, St.~John’s, NL, A1C~5S7, Canada. \href{mailto:dapike@mun.ca}{{\fontfamily{cmtt}\selectfont dapike@mun.ca}}}} \and \textit{Shahriyar Pourakbar Saffar\thanks{Department of Mathematics and Statistics, Memorial University of Newfoundland, St.~John’s, NL, A1C~5S7, Canada. \href{mailto:spourakbarsa@mun.ca}{{\fontfamily{cmtt}\selectfont spourakbarsa@mun.ca}}}}}

\begin{document}

\maketitle

\begin{abstract}
A cycle system of order $n$ is a decomposition of the edges of the complete graph $K_n$ into cycles of a fixed length. A cycle system is said to be $k$-colourable if we can assign $k$ colours to its vertices so that no cycle is monochromatic. A $k$-colourable cycle system is uniquely $k$-colourable if its colouring is unique up to the permutation of colour classes. In this paper, we construct uniquely $2$-colourable $4$-cycle systems of order $n$ for all admissible  $n\geq 49$, and also uniquely $2$-colourable $4$-cycle decompositions of $K_n - I$, for all admissible $n \geq 50$. These constructions contribute to the broader study of uniquely colourable cycle systems and open new directions for future research.
\end{abstract}
\vskip-0.1cm
\MSC{05C51, 05C15, 05B30}
\par\vskip.5cm

\section{Introduction}

We begin with a few well-known definitions. For $m>2$, an \textit{$m$-cycle decomposition} of a graph $G$ is a family of cycles of length $m$ such that every edge of $G$ appears in exactly one of the cycles. The necessary and sufficient conditions for the existence of an $m$-cycle decomposition of the complete graph $K_n$, also called an \textit{$m$-cycle system of order $n$}, are that either $n= 1$, or $n\geq m$, $n$ is odd, and $m$ divides $ \frac{n(n-1)}{2}$ \cite{AlspachGavlas2001,MR2012428, MR2700989,MR1871681}. More specifically, a $4$-cycle system of order $n$ exists if and only if $n\equiv 1\ (\mathrm{mod}\ 8)$. When $I$ is a 1-factor on $n$ vertices, it is also known that an $m$-cycle decomposition of the cocktail party graph $K_n- I$ (i.e., the graph obtained by removing the edges of $I$ from $K_n$) exists if and only if either $n=2$, or $n\geq m$, $n$ is even, and $m$ divides $ \frac{n(n-2)}{2}$ \cite{AlspachGavlas2001, MR2700989,MR1871681}. Consequently, a $4$-cycle decomposition of the graph $K_n- I$ exists if and only if $n\geq 2$ is even. We use the adjective ``\textit{admissible}'' for any order $n$ that satisfies the relevant conditions for the graph in question.

For a positive integer $k$, a \textit{weak $k$-colouring} of an $m$-cycle decomposition, namely $\cD$, of the graph $G$ is a mapping $\varphi:V(G)\rightarrow \{c_1,\ldots,c_k\}$ that assigns one of the colours $c_1,c_2,\ldots,c_k$ to each vertex of $G$ in such a way that no $m$-cycle in the decomposition is monochromatic, i.e., for all $C\in\cD$, $\lvert\varphi(V(C))\rvert >1$.
The sets $\varphi^{-1}(c_1),\varphi^{-1}(c_2),\ldots,\varphi^{-1}(c_k)$ are known as the \textit{colour classes} of the colouring $\varphi$, and $\varphi$ partitions the vertex set of $G$ into its nonempty colour classes.
We call $\varphi$ a \textit{balanced colouring} when, for all integers $i$ and $j$, $\big\lvert\lvert\varphi^{-1}(c_i)\rvert -\lvert\varphi^{-1}(c_j)\rvert\big\rvert \leq 1$. If a weak $k$-colouring of $\cD$ exists, we call that decomposition \textit{weakly $k$-colourable}, and if $k$ is the smallest number for which $\cD$ is weakly $k$-colourable, we say $\cD$ is \textit{weakly $k$-chromatic} and the \textit{weak chromatic number} of $\cD$, denoted by $\chi(\cD)$, is $k$. Throughout this paper we only consider weak colourings; therefore, for simplicity, the adjective ``weak'' and the adverb ``weakly'' will be omitted hereafter.

The history of colouring cycle systems starts with Steiner triple systems, which also happen to be $3$-cycle systems. In 1982, de Brandes, Phelps, and Rödl proved that for every integer $k\geq 3$, one can construct a $k$-chromatic Steiner triple system of any sufficiently large admissible order $n$ \cite{deBrandesPhelpsRodl1982}. Afterwards, in 1999, Milici and Tuza showed the existence of a 2-chromatic $m$-cycle system for all $m>3$ and for infinitely many admissible orders \cite{MiliciTuza1996}. In the following decade, Burgess and Pike proved the existence of an $m$-cycle system with arbitrary chromatic number when $m$ is even \cite{MR2185517, MR2378931}. Furthermore, in 2010 Horsley and Pike extended these results by establishing that a $k$-chromatic $m$-cycle system exists for all $(k,m)\neq (2,3)$ and for every large admissible order \cite{MR2677684}.

One may notice that the count of different ways to colour a cycle decomposition does not increase when we reduce the number of colours. Let $\cD$ be an $m$-cycle decomposition of a graph $G$ on $n>1$ vertices, and $\varphi:V(G)\rightarrow \{c_1,\ldots,c_k\}$ be a $k$-colouring of $\cD$ with the colour classes $C_1,\dots,C_k$. The decomposition $\cD$ is said to be \textit{uniquely $k$-colourable} if any other $k$‑colouring of $\cD$, namely $\psi$, yields the same family of colour classes, i.e., there exists a permutation map $\sigma_{\psi}:\{c_1,\ldots,c_k\}\rightarrow \{c_1,\ldots,c_k\}$ such that $\varphi=\sigma_{\psi}\circ\psi$. Equivalently, no two distinct $k$‑colourings of the $k$‑colourable $m$-cycle decomposition $\cD$ can produce different partitions of $V(G)$ via their colour classes. It is worth mentioning that since $\lvert V(G)\rvert >1$ and $m>2$, if $\cD$ is a uniquely $k$-colourable $m$-cycle decomposition of $G$, then this decomposition is also $k$-chromatic. If $k\geq n$, it is easy to see that $\cD$ is neither uniquely $k$-colourable, nor $k$-chromatic. If $k<n$, observe that any $k$-colouring of $\cal{D}$ that is not surjective cannot be unique (because we can re-colour any vertex of a colour class of size at least 2 with any unused colour, thereby obtaining a different $k$-colouring) and so any unique colouring of $\cal{D}$ must use all $k$ colours and therefore is a surjective colouring.  It follows that if $\cal{D}$ is uniquely $k$-colourable, then $\cal{D}$ must be $k$-chromatic (because any $(k-1)$-colouring would correspond to a $k$-colouring that is not surjective, indicating that there is a different $k$-colouring other than the supposedly unique one).

Any nontrivial Steiner triple system has chromatic number at least 3 \cite{MR309754} and thus there are no uniquely 2-colourable Steiner triple systems of order $n>3$. In 2003, Forbes built upon a few known uniquely $3$‑colourable Steiner triple systems mentioned in \cite{ColbournHaddadLinek1997,deBrandesPhelpsRodl1982,MR1961750} and showed that for every admissible order $n\ge25$, there exists a uniquely $3$‑colourable Steiner triple system with a balanced colouring and also one with no balanced colouring \cite{MR1953280}.
Every Steiner triple system of order $n$, $3<n\leq 19$, is known to be 3-chromatic \cite{MR2661401,deBrandesPhelpsRodl1982}. As mentioned in \cite{MR1961750}, there are no uniquely $3$‑colourable Steiner triple systems with an order smaller than 19. For those of order 19, which were enumerated in 2004 by Kaski and {\"O}sterg{\aa}rd \cite{MR2059752}, we have computationally verified that none of them are uniquely $3$‑colourable. As there are 14,796,207,517,873,771 non-isomorphic Steiner triple systems of order 21 \cite{MR4632477}, it is currently infeasible to check their 3-colourings.

That is where the prior knowledge of uniquely colourable cycle decompositions ends. In this paper, we push it further by affirmatively answering the questions ``Do any uniquely 2‑colourable 4‑cycle systems exist?''\ and ``Do any uniquely 2-colourable 4-cycle decompositions of cocktail party graphs exist?''. To achieve these goals, we first introduce the concept of an alternating colouring.

Let $G$ be a multipartite graph with parts $V_1,V_2,\ldots,V_h$\ such that the vertices within each part are totally ordered, where we assume each $V_i$ is the chain of vertices $v_{i,1},v_{i,2},\ldots,v_{i,\lvert V_i\rvert}$, i.e., the ordering is such that $v_{i,j}<v_{i,j+1}$ for all $j$, $1\leq j<\lvert V_i\rvert$. Suppose $\cD$ is an $m$-cycle decomposition of $G$. We call a $2$-colouring $\varphi$ of $\cD$ a \textit{2-alternating-colouring}, or just \textit{alt-colouring}, if within each part $V_i$ and for all $j\in\{1,2,\ldots,\lvert V_i\rvert -1\}$, $\varphi( v_{i,j})\neq \varphi(v_{i,j+1})$, i.e., every two consecutive vertices have different colours. We call an alt-colouring $\varphi$ of $\cD$ a \textit{super-2-alternating-colouring}, or just \textit{super-alt-colouring}, if for all $i\in\{1,\ldots,h-1\}$, $\varphi( v_{i,\lvert V_i\rvert})\neq \varphi(v_{i+1,1})$, i.e., the last vertex of each part and the first vertex of the next part have different colours. Figure~\ref{fig:Alt-colouring} illustrates two alt-colourings, one of which is a super-alt-colouring. A decomposition $\cD$ is called \textit{alt-colourable} if it has an alt-colouring. A decomposition $\cD$ is called \textit{exclusively alt-colourable} if $\cD$ admits a super-alt-colouring and, moreover, every $2$-colouring of $\cD$ is an alt-colouring.
We call a $2$-colouring $\varphi$ of $\cD$ a \textit{partially 2-alternating-colouring}, or just \textit{partially alt-colouring}, if within at least one $V_i$ with at least two vertices and for all $j\in\{1,2,\ldots,\lvert V_i\rvert -1\}$, $\varphi( v_{i,j})\neq \varphi(v_{i,j+1})$, i.e., within at least one part of $G$ with at least two vertices, every two consecutive vertices have different colours. A decomposition $\cD$ is called \textit{partially alt-colourable} if it has a partially alt-colouring. A decomposition $\cD$ is called \textit{exclusively partially alt-colourable} if every $2$-colouring of $\cD$ is a partially alt-colouring. The complete tripartite graph with parts of size 2, decomposed into its three complete bipartite sections (see Figure~\ref{fig:E-Alt-colouring}) is the smallest nontrivial example of an exclusively partially alt-colourable 4-cycle decomposition.
\begin{figure}
        \centering
\begin{tikzpicture}[scale=1.2]

\begin{scope}

  \foreach \i in {1,...,3} {
    \foreach \j in {\i,...,3} {
    \foreach \k in {1,...,4} {
    \foreach \g in {1,...,4} {
    \draw[gray!60] ($2*(\k,0)-(2,\i)$)--($2*(\g,-0.5)-(2,\j)$);
  }
  }
  }
  }
  
	\foreach \i in {1,3} {
  \foreach \j in {3} {
  \fill ($2*(\i,0)-(2,\j)$) circle (7pt);
  \node[white,scale=0.75]  at ($2*(\i,0)-(2,\j)$) {$v_{\j,\i}$};
	}
	}  
	
	\foreach \i in {2,4} {
  \foreach \j in {3} {
  \fill[white] ($2*(\i,0)-(2,\j)$) circle (7pt);
  \node[scale=0.75]  at ($2*(\i,0)-(2,\j)$) {$v_{\j,\i}$};
	}
	}  
	
	\foreach \i in {2,4} {
  \foreach \j in {1,2,4} {
  \fill ($2*(\i,0)-(2,\j)$) circle (7pt);
  \node[white,scale=0.75]  at ($2*(\i,0)-(2,\j)$) {$v_{\j,\i}$};
	}
	}  
	
	\foreach \i in {1,3} {
  \foreach \j in {1,2,4} {
  \fill[white] ($2*(\i,0)-(2,\j)$) circle (7pt);
  \node[scale=0.75]  at ($2*(\i,0)-(2,\j)$) {$v_{\j,\i}$};
	}
	}

  \foreach \i in {1,...,4} {
  \foreach \j in {1,...,4} {
  \draw ($2*(\i,0)-(2,\j)$) circle (7pt);
	}
	}

  \node[above=6pt] at (3,-5.25) {An alt-colouring of $\mathcal{D}$};
\end{scope}

\begin{scope}[xshift=7cm]

  \foreach \i in {1,...,3} {
    \foreach \j in {\i,...,3} {
    \foreach \k in {1,...,4} {
    \foreach \g in {1,...,4} {
    \draw[gray!60] ($2*(\k,0)-(2,\i)$)--($2*(\g,-0.5)-(2,\j)$);
  }
  }
  }
  }

  	\foreach \i in {1,3} {
  \foreach \j in {1,...,4} {
  \fill ($2*(\i,0)-(2,\j)$) circle (7pt);
  \node[white,scale=0.75]  at ($2*(\i,0)-(2,\j)$) {$v_{\j,\i}$};
	}
	}  
	
	\foreach \i in {2,4} {
  \foreach \j in {1,...,4} {
  \fill[white] ($2*(\i,0)-(2,\j)$) circle (7pt);
  \node[scale=0.75]  at ($2*(\i,0)-(2,\j)$) {$v_{\j,\i}$};
	}
	}

  \foreach \i in {1,...,4} {
  \foreach \j in {1,...,4} {
  \draw ($2*(\i,0)-(2,\j)$) circle (7pt);
	}
	}

  \node[above=6pt] at (3,-5.25) {A super-alt-colouring of $\mathcal{D}$};
\end{scope}

\end{tikzpicture}

        \caption{Examples of different alt-colourings. This figure shows a complete quadripartite graph with parts of size 4; here we suppose $\cD$ is a 4-cycle decomposition of this graph, and the colourings shown in the figure are 2-colourings of $\cD$.}
        \label{fig:Alt-colouring}
\end{figure}
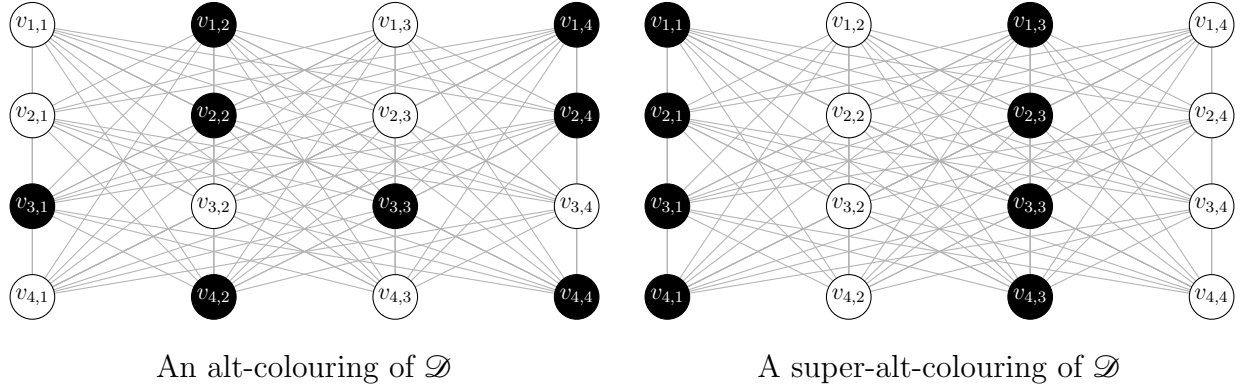

\begin{figure}
        \centering
\begin{tikzpicture}[scale=1.75]

\coordinate (v11) at (70:2);
\coordinate (v12) at (110:2);

\coordinate (v21) at (-170:2);
\coordinate (v22) at (-130:2);

\coordinate (v31) at (-50:2);
\coordinate (v32) at (-10:2);

\draw[thick,dash dot] (v21)--(v32)--(v22)--(v31)--cycle;
\draw[thick] (v21)--(v12)--(v22)--(v11)--cycle;
\draw[thick,dotted] (v11)--(v32)--(v12)--(v31)--cycle;

\fill[white] (v11) circle (3pt);
\fill[white] (v12) circle (3pt);
\fill (v21) circle (3pt);
\fill[white] (v22) circle (3pt);
\fill (v31) circle (3pt);
\fill (v32) circle (3pt);
  
  \foreach \i in {1,2,3} {
  \foreach \j in {1,2} {
  \draw (v\i\j) circle (3pt);
	}
	}

\end{tikzpicture}

        \caption{An exclusively partially alt-colourable 4-cycle decomposition of the complete tripartite graph with parts of size 2; at least one part must always be non-monochromatic.}
        \label{fig:E-Alt-colouring}
\end{figure}
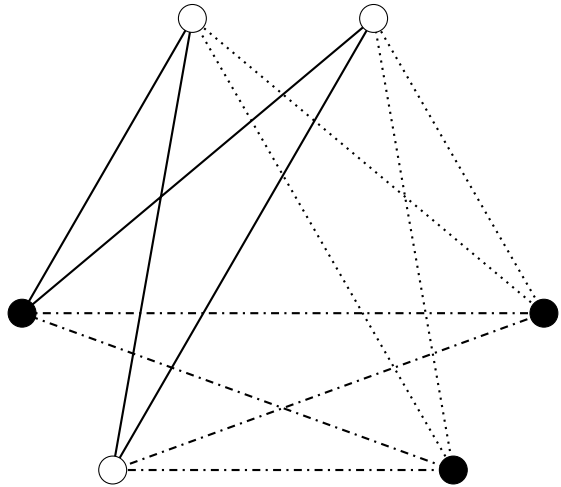

Although being exclusively alt-colourable does not necessarily imply possessing a unique 2-colouring, an exclusively alt-colourable cycle decomposition yields predictable colour patterns, which we will use to our advantage.
We will constructively prove the existence of exclusively alt-colourable 4-cycle decompositions of several complete multipartite graphs and so we have the following theorem.

\begin{restatable}{theorem}{EAfourCD}
\label{thm: EA4CD}
If $h\geq 6$ and for each $i\in\{1,\ldots,h\}$, $\l_i\geq 2$ is an integer, then there exists an exclusively alt-colourable 4-cycle decomposition of $K_{4{\l_1},\ldots,4{\l_h}}$, the complete $h$-partite graph with parts of size $4\l_1,\ldots,4\l_h$.
\end{restatable}

Embedding exclusively alt-colourable 4-cycle decompositions of complete multipartite graphs into 4-cycle decompositions of complete graphs and cocktail party graphs will enable us to prove the main results of this paper, which we now state.

\begin{restatable}{theorem}{UtwofourCS}
\label{thm: U24CS}
For all admissible $n \geq 49$, there exists a uniquely 2-colourable 4-cycle system of order $n$.
\end{restatable}

\begin{restatable}{theorem}{UtwofourCD}
\label{thm: U24CD}
For all admissible $n \geq 50$, there exists a uniquely 2-colourable 4-cycle decomposition of the cocktail party graph $K_n- I$.
\end{restatable}

Section~\ref{sec2} of this paper is entirely dedicated to the proof of Theorem~\ref{thm: EA4CD}. The methods and notation used in the second section of this paper are independent of Section~\ref{sec3}, where we prove Theorem~\ref{thm: U24CS} and Theorem~\ref{thm: U24CD}, so the reader can read the latter section as a standalone, provided that they accept the validity of Theorem~\ref{thm: EA4CD}. Lastly, Section~\ref{sec4} discusses future directions after this paper.

\section{Exclusively alt-colourable 4-cycle decompositions}
\label{sec2}

In this section we prove Theorem~\ref{thm: EA4CD}. To do so, we define a total ordering for the parts of a complete multipartite graph and we will explicitly form an exclusively alt-colourable 4-cycle decomposition for that graph.

Let us start by introducing an ordered set for the labelling of each part. Consider the set of triples with labels in $\Z_{\l }\times\Z_2\times\Z_2$ for a positive integer $\l$. We denote this set of points by $L_{\l,2,2}$. For the sake of clarity, for any positive integer $n$, $\Z_n$ consists of the elements $0,\ldots,(n-1)$ (with the natural order $0<1<\cdots<(n-2)<(n-1)$), but the usual algebraic ring operations of $\Z_n$ can be used if necessary (for example, $(n-1)+1=0$).
We set the primary order of $L_{\l,2,2}$ to be the lexicographic order that inherits the order of $\Z_{\l }$ and the order of $\Z_2$. So, for all $(a_1,a_2,a_3)\in L_{\l,2,2}$ and $(b_1,b_2,b_3)\in L_{\l,2,2}$, $(a_1,a_2,a_3)\prec(b_1,b_2,b_3)$ if and only if there exists $i\in\{1,2,3\}$ such that $a_i<b_i$, and $a_j=b_j$ whenever $j<i$. We denote the \textit{successor} of $x\in L_{\l,2,2}$ by  $x^{\triangleright}$, which is the element of $ \L_{\l,2,2}$ such that $x$ and $x^{\triangleright}$ are consecutive and $x\prec x^{\triangleright}$. We also adopt the convention $(\l-1,1,1)^{\triangleright}=(0,0,0)$. For a subset $A\subseteq \L_{\l,2,2}$, we denote the set $\{x^\triangleright\mid x\in A\}$ by $A^\triangleright$.

Let $A=\{x,y\}$ be a pair of distinct points in $L_{\l,2,2}$.
If $y=x^{\triangleright}$ and $x\in \Z_{\l }\times\Z_2\times\{ 0\}$, then we call $A$ an \textit{$\a$-set}.
If $y=x^{\triangleright}$ and $y\in \Z_{\l }\times\Z_2\times\{ 0\}$, then we call $A$ a \textit{$\b$-set}.
If $y-x=(0,1,0)$, i.e., $A=\{(i,0,c),(i,1,c)\}$ for some $i\in\Z_{\l}$ and $c\in\Z_2$, then we call $A$ a \textit{$\g$-set}. We denote the $\g$-set $\{(i,0,c),(i,1,c)\}$ by $\G_{i,c}$.

One can observe that since we have an even number of points in $\L_{\l,2,2}$, the family of all $\a$-sets partitions $\L_{\l,2,2}$ into 2-subsets. This family forms what we call the \textit{$\a$-partition} of $\L_{\l,2,2}$. In the same fashion, we define the \textit{$\b$-partition} and the \textit{$\g$-partition} of $\L_{\l,2,2}$ as the family of all $\b$-sets and the family of all $\g$-sets, respectively, and they both partition $\L_{\l,2,2}$ into 2-subsets.

Now, view the elements of $\L_{\l,2,2}$ as vertices and suppose $\L_{\l,2,2}$ is a coclique on $4\l$ vertices, totally ordered as we have described. As we will see in Lemmas \ref{lem: a b TPT} and \ref{lem: a b g alt}, a 2-colouring of the vertices of $\L_{\l,2,2}$, namely $\varphi$, can interact curiously with some partitions of $\L_{\l,2,2}$. With respect to $\varphi$ we call a pair of vertices $\{x,y\}$ a \textit{twin-pair of type c} or a $\TPT{c}^{\varphi}$ if $\varphi(x)=\varphi(y)=c$. When the colouring is fixed, we instead use $\TPT{c}$ and when $\{x,y\}$ is an $\a$-set ($\b$-set, $\g$-set, resp.), we may use the term $\a$-twin-pair ($\b$-twin-pair, $\g$-twin-pair, resp.). We call $\varphi$ an \textit{$\a$-free} (\textit{$\b$-free}, \textit{$\g$-free}, resp.) colouring if it contains no $\a$-twin-pairs ($\b$-twin-pairs, $\g$-twin-pairs, resp.)\ of any type. With this terminology, we may assume $\L_{\l,2,2}$ is a single part of a multipartite graph and characterise an alt-colouring on the vertices of this graph as a 2-colouring that is both $\a$-free and $\b$-free, i.e., for all $x\in \L_{\l,2,2}$, $\varphi(x)\neq\varphi(x^{\triangleright})$. This assertion mostly follows by definition, but to check whether or not it is also true that $\varphi((\l-1,1,1))\neq\varphi((0,0,0))$ when $\varphi$ is an alt-colouring, we use the following lemma.

\begin{lemma}
    Suppose $\varphi:\L_{\l,2,2}\to \Z_2$ is a 2-colouring of the vertices of $\L_{\l,2,2}$. For every $c\in\Z_2$, let $a_{c}$ be the number of $\a$-$\TPT{c}$ and $b_{c}$ be the number of $\b$-$\TPT{c}$ in $\L_{\l,2,2}$. Then $a_{0}+b_{1}=a_{1}+b_{0}$.
    \label{lem: a b TPT}
\end{lemma}

\begin{proof}
    Let $C_0$ be the set of all vertices with colour $0$. We have a total of $\frac{4\l}{2}-(a_0+a_1)$ $\a$-sets and $\frac{4\l}{2}-(b_0+b_1)$ $\b$-sets that are not twin-pairs. Note that every non-twin-pair 2-subset of $\L_{\l,2,2}$ must contain exactly one vertex with colour 0. Hence, $$2(a_{0})+(2\l-(a_{0}+a_{1}))=\lvert C_0\rvert=2(b_{0})+(2\l-(b_{0}+b_{1}))\Rightarrow a_{0}+b_{1}=a_{1}+b_{0}.$$

\vspace*{-1.5\baselineskip}
\end{proof}

If $\varphi$ is an alt-colouring, then there are no $\a$-twin-pairs and no $\b$-twin-pairs except possibly $\{(\l-1,1,1),(0,0,0)\}$, which cannot be the only $\b$-twin-pair; otherwise, we would have $a_{0}+b_{1}\neq a_{1}+b_{0}$. As a result, the aforementioned characterisation of an alt-colouring is valid and we may use it hereafter. In fact, Lemma~\ref{lem: a b g alt} contains additional characterisations for an alt-colouring of $\L_{\l,2,2}$.

\begin{lemma}

	Suppose $\varphi$ is a 2-colouring of the vertices of $\L_{\l,2,2}$. The following statements are equivalent.
	\begin{enumerate}
	\item $\varphi$ is an alt-colouring.
	\item For all $x\in \L_{\l,2,2}$, $\varphi(x)\neq\varphi(x^\triangleright)$.
	\item $\varphi$ is both $\a$-free and $\b$-free.
	\item $\cX_0=\Z_{\l}\times\Z_2\times\{0\}$ and $\cX_1=\Z_{\l}\times\Z_2\times\{1\}$ are the two colour classes of $\varphi$.
	\end{enumerate}
    \label{lem: a b g alt}
\end{lemma}

\begin{proof}

We have already covered the equivalency of the first three statements. We prove the second and fourth statements are equivalent. For $i\in\{0,1\}$, if $x\in\cX_i$, then $x^\triangleright\not\in\cX_i$ and $(x^\triangleright)^\triangleright\in\cX_i$. If for a set $A\subseteq \L_{\l,2,2}$, $A=(A^\triangleright)^\triangleright$ and $x\in A$, then $(x^\triangleright)^\triangleright\in A$. Iteratively, as a result, the only sets with the property that $A=(A^\triangleright)^\triangleright$ are $\emptyset$, $\cX_0$, $\cX_1$, and $\L_{\l,2,2}$.
Suppose $\{A,B\}$ is a partition of $\L_{\l,2,2}$ such that $A^\triangleright=B$. Let $C$ be a set such that $C^\triangleright=A$. If $x\in A\cap C$, then, contradictively, $x^\triangleright\in A^\triangleright\cap C^\triangleright=B\cap A=\emptyset$. Therefore, $C\subseteq  A^c=B$ and since $\lvert C \rvert = \lvert C^\triangleright \rvert = \lvert A \rvert = \lvert A^\triangleright \rvert = \lvert B \rvert$, $C=B$ and so $A=(A^\triangleright)^\triangleright$.
Hence, $\{\cX_0,\cX_1\}$ is the unique partition of $\L_{\l,2,2}$ such that $\cX_0^\triangleright=\cX_1$. This property precisely characterises the colour classes described in the second statement. The uniqueness yields the equivalency.
\end{proof}

This lemma will give us a powerful tool, because if $\varphi$ is an alt-colouring, then all $\g$-sets will be twin-pairs and for all $i\in\Z_{\l}$, $\G_{i,0}$ and $\G_{i,1}$ are $\g$-twin-pairs of different types.

Since we want to use the join of some copies of the coclique $\L_{\l,2,2}$ to construct a part-wise totally ordered complete multipartite graph, we may add an additional variable to $\L_{\l,2,2}$ so we will be able to distinguish between different parts of the constructed complete multipartite graph. We use the new notation $\L^i_{\l,2,2}$ for $\{i\}\times\Z_{\l}\times\Z_2\times\Z_2\cong\{i\}\times\L_{\l,2,2}$.

A complete multipartite graph can be decomposed into complete bipartite subgraphs, each of which is induced by a pair of maximal cocliques. So, to decompose a multipartite graph into 4-cycles, we may start with these bipartite subgraphs.  A complete bipartite graph is 4-cycle decomposable, i.e., it can be decomposed into 4-cycles, if and only if each of its parts has an even number of vertices \cite{MR609596}. If even-sized subsets $X_1$ and $X_2$ partition a part of a 4-cycle decomposable complete bipartite graph $G$, and if $\cD_1$ and $\cD_2$ are 4-cycle decompositions of induced subgraphs on $V(G)\setminus X_2$ and $V(G)\setminus X_1$, respectively, then $\cD_1\cup\cD_2$ decomposes $G$ into 4-cycles. But each of these induced subgraphs is again a complete bipartite graph. So we may continue breaking the parts of these smaller bipartite graphs in the same manner until we end up with subgraphs isomorphic to $K_{2,2}$---which is a 4-cycle. In particular, let $G=K_{m,n}$ be a 4-cycle decomposable complete bipartite graph with parts $A$ and $B$ and suppose $\lvert A \rvert = m$ and $\lvert B \rvert = n$. Let $\cA$ be a partition of the vertices of $A$ into 2-subsets and $\cB$ be a partition of the vertices of $B$ into 2-subsets. The family of 4-cycles $$\mathrm{BD}(\mathcal{A},\mathcal{B})=\left\{(a,b,a',b')\mid \{a,a'\}\in\mathcal{A},\{b,b'\}\in\mathcal{B}\right\}$$ is a 4-cycle decomposition of $G$.

For a positive integer $\l$, we define $\BG_{\l,2,2}$ as the complete bipartite graph with parts $L^0_{\l,2,2}$ and $L^1_{\l,2,2}$. Since $\lvert \L^0_{\l,2,2}\rvert=\lvert \L^1_{\l,2,2}\rvert=\lvert \L_{\l,2,2}\rvert=4\l$, $\BG_{\l,2,2}$ is 4-cycle decomposable and indeed we will introduce three different decompositions for this graph. Suppose $\cA_0$ and $\cA_1$ are the $\a$-partitions, $\cB_0$ and $\cB_1$ are the $\b$-partitions, and $\cC_0$ and $\cC_1$ are the $\g$-partitions for $L^0_{\l,2,2}$ and $L^1_{\l,2,2}$, respectively. We call $\mathrm{BD}(\mathcal{A}_0,\mathcal{A}_1)$ the \textit{$\a$-decomposition} of $\BG_{\l,2,2}$ and $\mathrm{BD}(\mathcal{B}_0,\mathcal{B}_1)$ the \textit{$\b$-decomposition} of $\BG_{\l,2,2}$. If $\l >1$, then we call $\mathrm{BD}(\{\G_{0,0},\G_{0,1}\},\mathcal{A}_1)\cup\mathrm{BD}(\cC_0\setminus\{\G_{0,0},\G_{0,1}\},\mathcal{B}_1)$ the \textit{$\g$-decomposition} of $\BG_{\l,2,2}$. Every cycle in the $\a$-decomposition, the $\b$-decomposition, and the $\g$-decomposition of $\BG_{\l,2,2}$ contains two consecutive vertices of $L^1_{\l,2,2}$. Therefore, by Lemma~\ref{lem: a b g alt}, no alt-colouring on $V(\BG_{\l,2,2})$ can admit a monochromatic cycle in any of these decompositions. 

\begin{remark}
The $\a$-decomposition, the $\b$-decomposition, and the $\g$-decomposition of \linebreak
$\BG_{\l,2,2}$ are all not only alt-colourable, but each can be coloured by all possible alt-colourings.
\label{rem: a b g alt}
\end{remark}

Let $\cD$ be a 4-cycle decomposition of $\BG_{\l,2,2}$, and let $\varphi:V(\mathrm{BG}_{\l,2,2})\rightarrow\Z_2$ be a 2-colouring of $\cD$. If $(a_1,a_2,a_3,a_4)\in\cD$ and for some $c\in\Z_2$, one of $\{a_1,a_3\}$ and $\{a_2,a_4\}$ is a $\TPT{c}$, then, to avoid a monochromatic cycle, the other cannot be a $\TPT{c}$. Consequently, the following lemma is a direct result of Lemma \ref{lem: a b g alt}.

\begin{lemma}

Suppose $\l>1$ is an integer, $\cD$ is the $\g$-decomposition of $\BG_{\l,2,2}$, and $\varphi$ is a 2-colouring of $\cD$. If the restriction of $\varphi$ to $L^0_{\l,2,2}$ is an alt-colouring, then $\varphi$ is also an alt-colouring.

\label{lem: g dec}
\end{lemma}

We can say $L^0_{\l,2,2}$, the first part of $\BG_{\l,2,2}$, ``exclusivises'' $L^1_{\l,2,2}$, the second part of $\BG_{\l,2,2}$, through the $\g$-decomposition. In fact, we will introduce a structure named the \textit{exclusiviser} that uses this property to expand an exclusively partially alt-colourable decomposition---and by which to restrict the colouring---into an exclusively alt-colourable decomposition.

Let $h$ and $\l$ be positive integers. We define $\MG_{h,\l,2,2}$ to be the complete $h$-partite graph with parts $\L^i_{\l,2,2}$ for all $i\in\Z_h$. For $h>1$, choose, arbitrarily, two integers $i\neq j$ from $\Z_{h}$; the induced subgraph of $\MG_{h,\l,2,2}$ on $\L^i_{\l,2,2}\cup \L^j_{\l,2,2}$, denoted by $\BG^{(i,j)}_{\l,2,2}$, is a bipartite subgraph naturally isomorphic to $\BG_{\l,2,2}$ where we map $\L^i_{\l,2,2}$ to the first part and $\L^j_{\l,2,2}$ to the second part of $\BG_{\l,2,2}$. As mentioned previously, we can decompose a multipartite graph into bipartite subgraphs. If $h\geq 1$ and $M$ is a tournament on the vertex set $\Z_h$, then the family $\{\BG^{(i,j)}_{\l,2,2}\mid (i\rightarrow j)\in M\}$ decomposes $\MG_{h,\l,2,2}$ into complete bipartite graphs. Hence, we can embed a 4-cycle decomposition for each of these complete bipartite graphs to construct a 4-cycle decomposition of $\MG_{h,\l,2,2}$. This decomposition is a specific case of a general structure in which we patch 4-cycle decompositions of smaller subgraphs together.

Suppose for positive integers $h$ and $r$, $\{h_1,h_2,\ldots,h_r\}$ is a set of positive integers with partial sums $\widetilde{h}_i=\Sigma_{j=1}^ih_j$ for all $i\in\{0,\ldots,r\}$, such that $\widetilde{h}_0=0$ and $\widetilde{h}_r=h$. Suppose for all $i\in\{1,\ldots,r\}$, $\cD_i$ is a 4-cycle decomposition of $\MG_{h_{i},\l,2,2}$. Note that if $h_i=1$, then $\cD_i=\emptyset$. Suppose $\df$ is a 4-cycle decomposition of $\BG_{\l,2,2}$ and let $\db$ be the 4-cycle decomposition of $\BG_{\l,2,2}$ where we replace every vertex $(0,i,j,c)$ with $(1,i,j,c)$ and vice versa in every cycle in $\df$ for all $(i,j,c)\in\Z_{\l}\times\Z_2\times\Z_2$. Suppose $M$ is a tournament with the vertex set $\{t_0,\ldots,t_{r-1}\}$. We will now construct a 4-cycle decomposition of $\MG_{h,\l,2,2}$, namely $\cD$, by patching $\cD_1,\ldots,\cD_{r}$ together through $M$ and using $\df$ as a bridge between them. To do so, we partition the parts of $\MG_{h,\l,2,2}$ into sets of vertices $U_i=\bigcup_{j=\widetilde{h}_i}^{\widetilde{h}_{i+1}-1}\L^j_{\l,2,2}$ for all $i\in\Z_r$. For all $i\in\{1,\ldots,r\}$, the subgraph of $\MG_{h,\l,2,2}$ induced on $U_{i-1}$ is an $h_{i}$-partite graph isomorphic to $\MG_{h_{i},\l,2,2}$; we can embed $\cD_i$ in $\cD$ and by doing so, we decompose this subgraph. Now we use $M$ as a guiding map to decompose the remaining $\BG^{(i,j)}_{\l,2,2}$ with $\df$. For $i'\neq j'$, let $\L^i_{\l,2,2}\subseteq U_{i'}$ and $\L^j_{\l,2,2}\subseteq U_{j'}$. If $(t_{i'}\rightarrow t_{j'})\in M$, then we use the natural isomorphism of $\BG^{(i,j)}_{\l,2,2}$ and embed $\df$ in $\cD$ for $\BG^{(i,j)}_{\l,2,2}$; otherwise, embed $\db$ for $\BG^{(i,j)}_{\l,2,2}$. This method gives us a 4-cycle decomposition of $\MG_{h,\l,2,2}$ that we denote by $\D_{h_1,\ldots,h_r}(\cD_1,\ldots,\cD_r;M;\df)$. Clearly, if $h=1$, then this decomposition will be an empty set. If $\df=\db$, then we do not need the tournament $M$ anymore and we can write $\D_{h_1,\ldots,h_r}(\cD_1,\ldots,\cD_r;\df)$ instead. Moreover, if $h_1=\cdots=h_r=1$, then $\cD_1=\cdots=\cD_r=\emptyset$ and we may use the notation $\D_{[h]}(M;\df)$ when $\df\neq \db$ and $\D_{[h]}(\df)$ when $\df= \db$. Note that $\D_{h_1,\ldots,h_r}(\cD_1,\ldots,\cD_r;M;\df)$ is alt-colourable if and only if $\cD_1,\ldots,\cD_r$ and $\df$ are each alt-colourable. Since the parts of $\MG_{h,\l,2,2}$ are of even size, an alt-colouring of a 4-cycle decomposition of $\MG_{h,\l,2,2}$ is a super-alt-colouring if and only if each part starts with the same colour, and thus the order in which we write the parts does not matter. Therefore, $\D_{h_1,\ldots,h_r}(\cD_1,\ldots,\cD_r;M;\df)$ admits a super-alt-colouring if and only if $\cD_1,\ldots,\cD_r$ and $\df$ each admits a super-alt-colouring.

Suppose the 4-cycle decompositions $\cD_1,\ldots,\cD_r$ are all alt-colourable, $M$ contains a directed Hamiltonian cycle $\mathcal{H}$, and $\df$ is the $\g$-decomposition of $\BG_{\l,2,2}$. If some of $\cD_1,\ldots,\cD_r$ are exclusively partially alt-colourable, then, by Lemma \ref{lem: g dec}, the directed Hamiltonian cycle $\mathcal{H}$ ensures that $\D_{h_1,\ldots,h_r}(\cD_1,\ldots,\cD_r;M;\df)$ cannot be 2-coloured by anything but alt-colourings. When $r=3$, $M$ is a directed Hamiltonian cycle, and $\df$ is the $\g$-decomposition of $\BG_{\l,2,2}$, then we define the \textit{exclusiviser} of $\cD$ as $\E_h(\cD)=\D_{h-2,1,1}(\cD,\emptyset,\emptyset;M;\df)$.

\begin{lemma}
If $h>2$ and $\cD$ is an exclusively partially alt-colourable 4-cycle decomposition of $\MG_{h-2,\l,2,2}$ which admits a super-alt-colouring, then $\E_h(\cD)$ is an exclusively alt-colourable 4-cycle decomposition of $\MG_{h,\l,2,2}$.

\label{lem: ex dec}
\end{lemma}

\begin{proof}
Note that by Remark~\ref{rem: a b g alt}, a $\g$-decomposition can be coloured with a super-alt-colouring. Since $\cD$ also admits a super-alt-colouring, then $\E_h(\cD)=\D_{h-2,1,1}(\cD,\emptyset,\emptyset;M;\df)$ can be coloured with a super-alt-colouring by construction.

The exclusivity of the alt-colouring is a direct consequence of Lemma \ref{lem: g dec}.
\end{proof}

It remains to find the required exclusively partially alt-colourable 4-cycle decomposition of $\MG_{h-2,\l,2,2}$.

Let $\l$ be a positive integer and let $\cD_{\a}$ and $\cD_{\b}$ be the $\a$-decomposition and the \linebreak
$\b$-decomposition of $\BG_{\l,2,2}$, respectively. Suppose $\varphi$ is a 2-colouring of $\cD_{\a}$ and its restriction to $\L^0_{\l,2,2}$ has at least one $\a$-$\TPT{0}$ and no $\b$-$\TPT{0}$. Then by Lemma \ref{lem: a b TPT}, the number of $\a$-$\TPT{1}$ for the restriction of $\varphi$ to $\L^0_{\l,2,2}$ must be a positive number.
Therefore, the restriction of $\varphi$ to $\L^1_{\l,2,2}$ must be $\a$-free to avoid monochromatic cycles. We can find similar results if we swap the colour 0 with colour 1, $\a$ with $\b$, or $\L^0_{\l,2,2}$ with $\L^1_{\l,2,2}$ (e.g., if the restriction of a 2-colouring of $\cD_{\b}$ to $\L^1_{\l,2,2}$ has at least one $\b$-$\TPT{1}$ and no $\a$-$\TPT{1}$, then its restriction to $\L^0_{\l,2,2}$ must be $\b$-free). These properties lead to the three following consequences:

(I) Suppose $\varphi$ is a 2-colouring of $\cD_{\a}$ ($\cD_{\b}$, resp.)\ and its restriction to $\L^0_{\l,2,2}$ is $\b$-free ($\a$-free, resp.)\ but not $\a$-free ($\b$-free, resp.). Then the restriction of $\varphi$ to $\L^1_{\l,2,2}$ is $\a$-free ($\b$-free, resp.).

(II) Suppose $\varphi$ is a 2-colouring of $\cD_{\a}$ ($\cD_{\b}$, resp.). If neither of the restrictions of $\varphi$ to $\L^0_{\l,2,2}$ and $\L^1_{\l,2,2}$ is $\a$-free ($\b$-free, resp.),\ then also neither of them is $\b$-free ($\a$-free, resp.).

(III) Suppose $\varphi$ is a 2-colouring of $\cD_{\a}$ ($\cD_{\b}$, resp.). If the restriction of $\varphi$ to $\L^0_{\l,2,2}$ has at least one $\a$-$\TPT{c}$ ($\b$-$\TPT{c}$, resp.),\ for some $c\in\Z_2$, and its restriction to $\L^1_{\l,2,2}$ is not $\a$-free ($\b$-free, resp.),\ then the restriction of $\varphi$ to $\L^0_{\l,2,2}$ contains at least one $\b$-$\TPT{c}$ ($\a$-$\TPT{c}$, resp.).

\begin{lemma}

$\cD_{\a\times\b}=\D_{2,2}\left(\D_{[2]}(\cD_{\a}),\D_{[2]}(\cD_{\a}\right);\cD_{\b})$ is exclusively partially alt-colourable.

\label{lem: a b dec}
\end{lemma}

\begin{proof}
We have to prove that for every 2-colouring of $\cD_{\a\times\b}$, the colouring of at least one part of $\MG_{4,\l,2,2}$ is both $\a$-free and $\b$-free.

We suppose $\varphi$ is a 2-colouring of $\cD_{\a\times\b}$ and $\varphi_i$ is the restriction of $\varphi$ to $\L^i_{\l,2,2}$ for $i\in\Z_4$. Let $\cD^{i,j}$ be the 4-cycle decomposition induced from $\cD_{\a\times\b}$ on $\BG^{(i,j)}_{\l,2,2}$ for all integers $i\neq j$ in $\Z_4$. Note that when $\{i,j\}=\{0,1\}$ or $\{i,j\}=\{2,3\}$, $\cD^{i,j}\cong \cD_{\a}$, and for other combinations of $i$ and $j$, $\cD^{i,j}\cong \cD_{\b}$.

For the sake of contradiction, suppose for all $i\in\Z_4$ the restriction $\varphi_i$ is not both $\a$-free and $\b$-free (i.e., each $\varphi_i$ has some $\a$-twin-pair or $\b$-twin-pair). Assume $\varphi_0$ is $\a$-free. Then, since $\cD^{0,2}$ and $\cD^{0,3}$ are $\b$-decompositions, by (I), $\varphi_2$ and $\varphi_3$ are $\b$-free, but (I) will guarantee that at least one of $\varphi_2$ and $\varphi_3$ must also be $\a$-free which is a contradiction. In fact, using the same logic, none of the restrictions can be $\a$-free; otherwise the hypothesis will be contradicted. Furthermore, by (II), for all $i\in\Z_4$, $\varphi_i$ is neither $\a$-free nor $\b$-free (i.e., each $\varphi_i$ has some $\a$-twin-pair and some $\b$-twin-pair).

Without loss of generality, we may assume $\varphi_0$ contains an $\a$-$\TPT{0}$. Therefore, by (III), $\varphi_0$ contains a $\b$-$\TPT{0}$. To avoid monochromatic cycles, $\varphi_1$ cannot contain the same type of $\a$-twin-pairs as $\varphi_0$, so it must contain an $\a$-$\TPT{1}$, and also by (III), it must contain a $\b$-$\TPT{1}$. Hence, to avoid monochromatic cycles in $\cD^{0,2}$ and $\cD^{1,2}$, $\varphi_2$ must be $\b$-free, which is a contradiction.
\end{proof}

More generally, whenever there exist distinct integers $i$ and $j$ in $\{1,\ldots,r\}$ such \linebreak
that $h_i>1$ and $h_j>1$, then $\D_{h_1,\ldots,h_r}\left(\D_{[h_1]}(\cD_{\a}),\ldots,\D_{[h_r]}(\cD_{\a});\cD_{\b}\right)$ is exclusively \linebreak
partially alt-colourable since the 4-cycle decomposition induced from \linebreak
$\D_{h_1,\ldots,h_r}\left(\D_{[h_1]}(\cD_{\a}),\ldots,\D_{[h_r]}(\cD_{\a});\cD_{\b}\right)$ on $\MG_{h,\l,2,2}\left[\L^{\widetilde{h}_{i-1}}_{\l,2,2}\cup\L^{\widetilde{h}_{i-1}+1}_{\l,2,2}\cup\L^{\widetilde{h}_{j-1}}_{\l,2,2}\cup\L^{\widetilde{h}_{j-1}+1}_{\l,2,2}\right]$ (the subgraph of $\MG_{h,\l,2,2}$ induced on the vertices in $\L^{\widetilde{h}_{i-1}}_{\l,2,2}\cup\L^{\widetilde{h}_{i-1}+1}_{\l,2,2}\cup\L^{\widetilde{h}_{j-1}}_{\l,2,2}\cup\L^{\widetilde{h}_{j-1}+1}_{\l,2,2}$) is isomorphic to $\cD_{\a\times\b}$. Now we have all we need to construct exclusively alt-colourable 4-cycle decompositions of complete multipartite graphs.

\begin{theorem}
If $h\geq 6$ and $\l\geq 2$ are integers, then there exists an exclusively alt-colourable 4-cycle decomposition of $K_{4\l,4\l,\ldots,4\l}$, the complete $h$-partite graph with parts of size $4\l$.
\end{theorem}

\begin{proof}
For arbitrary $h\geq 6$ and $\l\geq 2$, label the vertices in the same manner as $\MG_{h,\l,2,2}$ and let $\cD=\E_h(\cD')$ be the 4-cycle decomposition of this graph, where
$$\cD'=\D_{2,h-4}\left(\D_{[2]}(\cD_{\a}),\D_{[h-4]}(\cD_{\a});\cD_{\b}\right)$$
and $\cD_{\a}$ and $\cD_{\b}$ are the $\a$-decomposition and the $\b$-decomposition of $\BG_{\l,2,2}$, respectively.
The decomposition $\cD'$ is exclusively partially alt-colourable, since it contains $\cD_{\a\times\b}$ from Lemma~\ref{lem: a b dec}. Also, as a consequence of Remark~\ref{rem: a b g alt}, $\cD'$ can be coloured by all possible alt-colourings, and more specifically, a super-alt-colouring. Thus, $\cD$, the exclusiviser of $\cD'$, is exclusively alt-colourable by Lemma~\ref{lem: ex dec}.
\end{proof}

One may notice that the proof for this construction is independent of the regularity of each complete multipartite graph. For example, we can find an exclusively partially alt-colourable 4-cycle decomposition of $K_{4,4,4,8}$ in the same way by using the complete multipartite graph $\L_{1,2,2}\vee\L_{1,2,2}\vee\L_{1,2,2}\vee\L_{2,2,2}$, where $\vee$ is the join operation.
We can follow the construction of an exclusively alt-colourable 4-cycle decomposition step by step for any complete multipartite graph such that the number of vertices in each part is at least 8 and is divisible by 4. This generality leads to the first of our main results:

\EAfourCD*

\section{Uniquely 2-colourable 4-cycle decompositions}
\label{sec3}

In the previous section, we succeeded in proving Theorem~\ref{thm: EA4CD}, which allows us to construct an exclusively alt-colourable 4-cycle decomposition of a complete $h$-partite graph with parts $L^1,\ldots,L^h$, for all integers $h\geq 6$, where $L^i$ is totally ordered and $\lvert L^i\rvert=8$ for all $i\in\{1,\ldots,h\}$; we now denote this decomposition by $\EA_h$. Since $\EA_h$ is exclusively alt-colourable, any 2-colouring of $\EA_h$ partitions each $L^i$ into a fixed unique partition $\{A_i,B_i\}$ with its colour classes. Furthermore, for all $L^i$, $\lvert A_i\rvert=\lvert B_i\rvert=4$.

In the current section, we will carefully embed $\EA_h$ into a 4-cycle decomposition of a graph which contains the aforementioned complete $h$-partite graph as a subgraph, such that the larger decomposition is uniquely 2-colourable.
To do so, we will introduce some new terminology for this section; the number of colours used in these definitions and their properties is not necessarily 2 but an arbitrary positive integer, denoted by $k$, unless stated otherwise.
Suppose $S_1,\ldots,S_k$ are collections of coloured objects. We say $S_1,\ldots,S_k$ are \textit{homochromatic} if $S_1\cup\cdots\cup S_k$ is monochromatic, and we say $S_1,\ldots,S_k$ are \textit{pairwise heterochromatic}, or simply \textit{heterochromatic}, if each nonempty $S_i$, $i\in\{1,\ldots,k\}$, is monochromatic and has a different colour from each nonempty $S_j$, $j\neq i$.

Let $G$ be a graph and $\cD$ be a cycle decomposition of $G$.
Suppose $k$ and $r$ are positive integers and for all $i\in\{1,\ldots,r\}$ and $j\in\{1,\ldots,k\}$, $S_{i,j}$ are disjoint subsets of $V(G)$, some of which may be empty sets. We say $\cD$ is \textit{proto-anchored} to $\big\{\{S_{1,j}\}_{j=1}^k,\ldots,\{S_{r,j}\}_{j=1}^k\big\}$ if there is exactly one $k$-colouring of $\cD$, namely $\psi:V(G)\rightarrow\{1,\ldots,k\}$, such that for all $i$, $S_{i,1},\ldots,S_{i,k}$ are heterochromatic and $\psi(S_{1,j})\subseteq\{j\}$ for all $j\in\{1,\ldots,k\}$. When $r=1$, if $\cD$ satisfies the conditions to be proto-anchored, we say $\cD$ is \textit{anchored} to $\{S_{1,1},\ldots,S_{1,k}\}$ instead.
Alternatively, we call a family of disjoint subsets of $V(G)$, namely $\Psi$, an \textit{anchor} of order $\lvert\Psi\rvert$ for $\cD$ if $\cD$ is anchored to $\Psi$. It is obvious that $\EA_h$ is anchored to $\{S_A,\emptyset\}$, where $S_A$ is a system of distinct representatives for $\{A_i\}_{i=1}^h$. Note that for any decomposition, an anchor of order $k$ exists if and only if the decomposition is $k$-colourable; since the colour classes of a $k$-colouring are always heterochromatic, we may take this family as the anchor. Furthermore, we offer an alternative characterisation for a decomposition that is uniquely $k$-colourable in the following remark.

\begin{remark} 

If $\cD$ is a cycle decomposition, and the vertex sets $\{a_1\},\ldots,\{a_{k-1}\}$ are heterochromatic with respect to each $k$-colouring of $\cD$, then $\cD$ is uniquely $k$-colourable if and only if $\big\{\{a_1\},\ldots,\{a_{k-1}\},\emptyset\big\}$ is an anchor for $\cD$. Consequently, $\cD$ is uniquely 2-colourable if and only if $\big\{\{v\},\emptyset\big\}$ is an anchor for $\cD$ for any vertex $v$.

\end{remark}

Let $r>1$ be a positive integer. Suppose $H_1,\ldots,H_r$ are isomorphic subgraphs of a graph $G$ where for all distinct $i_1$ and $i_2$ in $\{1,\ldots,r\}$, $V(H_{i_1})\cap V(H_{i_2})=\bigcap_{i=1}^r V(H_i)$ and there exists an isomorphism $\theta_{i_1,i_2}:V(H_{i_1})\to V(H_{i_2})$ such that $\theta_{i_1,i_2}(s)=s$ for all $s\in \bigcap_{i=1}^r V(H_i)$. For the purpose of constructing a well-defined 4-cycle decomposition, if there is more than one such isomorphism between $H_{i_1}$ and $H_{i_2}$, then we fix one of those isomorphisms from $H_{i_1}$ to $H_{i_2}$ as $\theta_{i_1,i_2}$ and we correspondingly choose $\theta_{i_1,i_2}^{-1}$ as $\theta_{i_2,i_1}$. Let $S$ be the subgraph of $G$ with $V(S)=\bigcap_{i=1}^r V(H_i)$ and $E(S)=\bigcap_{i=1}^r E(H_i)$. Note that for every $i\in\{1,\ldots,r\}$, $S$ is the induced subgraph of $H_i$ on $V(S)$. Suppose $V(S)$ is nonempty, $S$ has a 4-cycle decomposition $\cD$, and $\cD$ is embedded in a 4-cycle decomposition $\widehat{\cD}_1$ of $H_1$; in other words, $\cD$, as a sub-decomposition of $\widehat{\cD}_1$ ($\cD\subseteq\widehat{\cD}_1$), decomposes $S$. For all $i\in\{2,\ldots,r\}$, let $\widehat{\cD}_i$ be the 4-cycle decomposition of $H_i$ that is obtained from $\widehat{\cD}_1$ through the isomorphism $\theta_{1,i}$. Note that $\cD\subseteq\widehat{\cD}_i$ for all $i$. We denote
$$\fd_r(\cD,\widehat{\cD}_1)=\bigcup_{i=1}^r\widehat{\cD}_i,$$
which is a 4-cycle decomposition of $\bigcup_{i=1}^rH_i$. Suppose $k$ is a positive integer and $P_{1,1},\ldots,P_{1,k}$ is a family of subsets of $V(G)$ such that the nonempty subsets partition $V(H_1)\setminus V(S)$. For all $i\in\{2,\ldots,r\}$, use the isomorphism $\theta_{1,i}$ to form the subsets $P_{i,1},\ldots,P_{i,k}$ from $P_{1,1},\ldots,P_{1,k}$, respectively (i.e., $P_{i,\ast}=\theta_{1,i}(P_{1,\ast})$); note that the nonempty subsets in $P_{i,1},\ldots,P_{i,k}$ partition $V(H_i)\setminus V(S)$. Suppose $\widehat{\cD}_1$ is anchored to $\{P_{1,1},\ldots,P_{1,k}\}$. Accordingly, for all $i\in\{2,\ldots,r\}$, $\widehat{\cD}_i$ is anchored to $\{P_{i,1},\ldots,P_{i,k}\}$. Since, for all $i$, $\{P_{i,1},\ldots,P_{i,k}\}$ must assign colours to $S$ in a unique way, the colours assigned to $P_{i,1},\ldots,P_{i,k}$ can be uniquely determined by the colouring of $S$ for each $i$ if we have enough colours in $S$.

\begin{lemma}
If the colouring of $\widehat{\cD}_1$, while anchored to $\{P_{1,1},\ldots,P_{1,k}\}$, produces at least $k-1$ colours in $S$, then $\fd_r(\cD,\widehat{\cD}_1)$ is proto-anchored to $\big\{\{P_{1,c}\}_{c=1}^k,\ldots,\{P_{r,c}\}_{c=1}^k\big\}$.

\label{lem: flower form}
\end{lemma}

\begin{proof}
For all $i\in\{1,\ldots,r\}$, $\widehat{\cD}_i$ is anchored to $\{P_{i,1},\ldots,P_{i,k}\}$ with the $k$-colouring $\psi_i:V(H_i)\rightarrow \{1,\ldots,k\}$ for which $\psi_i(P_{i,c})\subseteq\{c\}$ for all $i$ and $c$. Since for all $i\neq j$ in $\{1,\ldots,r\}$, the restriction of the isomorphism $\theta_{i,j}:H_i\to H_j$ to $S$ is the identity map, for all $v\in V(S)$, $\psi_1(v)=\psi_2(v)=\cdots=\psi_r(v)$. Now, let $\psi:V(\bigcup_{i=1}^rH_i)\rightarrow \{1,\ldots,k\}$ be a $k$-colouring of $\fd_r(\cD,\widehat{\cD}_1)$ such that $\psi(v)=\psi_i(v)$ for all $i$ and $v\in V(H_i)$. For all $i$, $P_{i,1},\ldots,P_{i,k}$ are heterochromatic and for all $c\in\{1,\ldots,k\}$, $\psi(P_{1,c})\subseteq\{c\}$; it only remains to check the uniqueness to prove $\fd_r(\cD,\widehat{\cD}_1)$ is proto-anchored to $\big\{\{P_{1,c}\}_{c=1}^k,\ldots,\{P_{r,c}\}_{c=1}^k\big\}$. Suppose $\psi':V(\bigcup_{i=1}^rH_i)\rightarrow \{1,\ldots,k\}$ is a $k$-colouring of $\fd_r(\cD,\widehat{\cD}_1)$ such that for all $i$, $P_{i,1},\ldots,P_{i,k}$ are heterochromatic and for all $c\in\{1,\ldots,k\}$, $\psi'(P_{1,c})\subseteq\{c\}$. Let $v\in V(S)$ and suppose for some $c$, $\psi'(v)=\psi_1(v)=c$. If $P_{1,c}$ is empty, then for all $i$, $P_{i,c}$ is empty. If $P_{1,c}$ is not empty, then $P_{1,c}$ and $\{v\}$ are homochromatic. But, for all $i$, $\psi'$ induces a $k$-colouring for $\widehat{\cD}_i$ with the same colour classes as $\psi_i$. Therefore, $P_{i,c}$ and $\{v\}$ are also homochromatic, and so, $\psi'(P_{1,c})=\psi'(P_{2,c})=\cdots=\psi'(P_{r,c})=\{c\}$.
Without loss of generality, suppose $\{1,\ldots,k-1\}\subseteq \psi'(V(S))$.
Hence, for all $i$ and for all $c\in\{1,\ldots,k-1\}$, $\psi'(P_{i,c})\subseteq\{c\}$.
However, for all $i$, $\psi'(P_{i,k})\cap\{1,\ldots,k-1\}=\emptyset$, so $\psi'(P_{i,k})\subseteq\{k\}$ and consequently, $\psi'=\psi$.
\end{proof}

We now return our attention to focus on $k=2$ colours. Since $S$ is nonempty, it must produce at least one colour for some vertex and the sufficient condition of Lemma~\ref{lem: flower form} holds.
If for some $h\geq 6$, $\EA_h$ decomposes a subgraph $G^*$ of $G$, then $V(G^*)$ can be partitioned into $\{A_i\}_{i=1}^h\cup\{B_i\}_{i=1}^h$ such that for all $i$, $A_i$ and $B_i$ are heterochromatic for any 2-colouring of $\EA_h$.
We will be able to use this fact to our advantage; when we find the appropriate $\fd_r(\cD,\widehat{\cD}_1)$, by adding it to $\EA_h$, we will construct a uniquely 2-colourable 4-cycle decomposition of $G$.

Considering that a $4$-cycle system of order $n$ exists only if $n\equiv 1\ (\mathrm{mod}\ 8)$, we let $G=K_{8h+1}$ for some positive integer $h$, where $$V(G)=\{s_1\}\cup\bigcup_{i=1}^h\{v_{i,1},\ldots,v_{i,8}\}.$$
Let $G^*$ be the complete $h$-partite graph with parts $\{v_{i,1},\ldots,v_{i,8}\}$ for all $i\in\{1,\ldots,h\}$.
Now, for all $i\in\{1,\ldots,h\}$, choose $H_i$ to be the induced subgraph of $G- G^*$ on $\{s_1,v_{i,1},\ldots,v_{i,8}\}$, and note that $H_i$ is a clique of order 9. Since all complete graphs of the same order are isomorphic, $H_1,\ldots,H_h$ are isomorphic and they would all be pairwise vertex-disjoint except that $s_1$ is a common vertex among all of them. For all $i\neq j$ in $\{1,\ldots,h\}$ we let $\theta_{i,j}:H_i\to H_j$ be the isomorphism where $\theta_{i,j}(v_{i,1})=v_{j,1},\ldots,\theta_{i,j}(v_{i,8})=v_{j,8}$ and $\theta_{i,j}(s_1)=s_1$.

\begin{lemma}

Let $H$ be the complete graph of order $9$. Suppose $P_{1}=\{\mathtt{a_1,\ldots,a_4}\}$, $P_{2}=\{\mathtt{b_1,\ldots,b_4}\}$, and $V(H)=P_{1}\cup P_{2}\cup \{\mathtt{s_1}\}$. Then there exists a 4-cycle decomposition $\widehat{\cD}$ of $H$, which is anchored to $\{P_{1}, P_{2}\}$.

\label{lem: flower form 1}

\end{lemma}

\begin{table}
\centering
\begin{tblr}{
  row{1} = {c},
  column{1} = {c},
  column{2} = {c},
  cell{2}{3} = {c},
  cell{3}{3} = {c},
  hlines,
  vlines,
}
$(H,S)$ & {$P_{1},P_{2}$} & $\cD$ & $\widehat{\cD}\setminus\cD$ \\

\hline

$(K_9,K_1)$    & {$\{\mathtt{a_1,\ldots,a_4}\},$\\ $\{\mathtt{b_1,\ldots,b_4}\}$}  & $\emptyset$ & {$\mathtt{(a_1,a_3,a_2,s_1)^*,(a_2,a_4,a_3,b_3),(a_3,b_1,a_4,b_4),}$\\
$\mathtt{(a_4,b_2,b_1,a_1),(b_1,s_1,b_2,a_2),(b_2,b_3,s_1,a_3),}$\\
$\mathtt{(s_1,b_4,b_3,a_4),(b_3,a_1,b_4,b_1),(b_4,a_2,a_1,b_2)}$} \\

{$(\overline{I_{5}},\overline{I'_{1}})$}    &  {$\{\mathtt{a_1,\ldots,a_4}\},$\\ $\{\mathtt{b_1,\ldots,b_4}\}$} & $\emptyset$ & {$\mathtt{(a_1,a_2,a_3,s_2)^*,(a_1,a_3,b_1,b_3),(s_1,a_2,s_2,b_2),}$\\
$\mathtt{(b_1,s_2,b_3,a_4),(b_1,b_2,b_3,s_1)^*,(a_1,s_1,a_3,b_4),}$\\
$\mathtt{(a_1,b_2,a_3,a_4),(b_1,a_2,b_3,b_4),(a_4,s_1,b_4,s_2),}$\\
$\mathtt{(a_2,a_4,b_2,b_4)}$} \\

{$(\overline{I_{6}},\overline{I'_{2}})$}   &  {$\{\mathtt{a_1,\ldots,a_4}\},$\\ $\{\mathtt{b_1,\ldots,b_4}\}$} & $\mathtt{(s_1,s_3,s_2,s_4)}$ & {$\mathtt{(a_1,a_2,a_3,s_2)^*,(s_1,a_3,s_3,b_2),(a_2,s_3,a_4,b_3),}$\\
$\mathtt{(a_3,a_4,b_1,s_4),(s_3,b_1,s_2,b_4)^*,(a_4,s_2,b_2,a_1),}$\\
$\mathtt{(b_1,b_2,b_3,s_1)^*,(s_2,b_3,s_4,a_2),(b_2,s_4,b_4,a_3),}$\\
$\mathtt{(b_3,b_4,a_1,s_3),(s_4,a_1,s_1,a_4)^*,(b_4,s_1,a_2,b_1),}$\\
$\mathtt{(a_2,a_4,b_2,b_4),(a_3,b_1,b_3,a_1)}$} \\

{$(\overline{I_{7}},\overline{I'_{3}})$}   &  {$\{\mathtt{a_1,\ldots,a_4}\},$\\ $\{\mathtt{b_1,\ldots,b_4}\}$} & {$\mathtt{(s_1,s_3,s_2,s_4),}$\\$\mathtt{(s_1,s_6,s_4,s_5),}$\\$\mathtt{(s_3,s_6,s_2,s_5)}$} & {$\mathtt{(a_1,a_2,a_3,s_2)^*,(s_1,a_3,s_3,b_2),(s_6,a_1,s_5,b_1),}$\\
$\mathtt{(a_3,a_4,b_1,s_4),(s_3,b_1,s_2,b_4)^*,(a_4,s_2,b_2,a_1),}$\\
$\mathtt{(b_1,b_2,b_3,s_1)^*,(s_2,b_3,s_4,a_2),(s_6,a_3,s_5,b_3),}$\\
$\mathtt{(b_3,b_4,a_1,s_3),(s_4,a_1,s_1,a_4)^*,(b_4,s_1,a_2,b_1),}$\\
$\mathtt{(a_2,a_4,b_2,b_4),(a_3,b_1,b_3,a_1),(a_2,s_3,a_4,s_6)^*,}$\\
$\mathtt{(b_2,s_4,b_4,s_5)^*,(s_6,b_2,a_3,b_4),(s_5,a_2,b_3,a_4)}$} \\

{$(\overline{I_{8}},\overline{I'_{4}})$}   &  {$\{\mathtt{a_1,\ldots,a_4}\},$\\ $\{\mathtt{b_1,\ldots,b_4}\}$} &  {$\mathtt{(s_1,s_8,s_2,s_7),}$\\$\mathtt{(s_1,s_4,s_8,s_6),}$\\$\mathtt{(s_3,s_2,s_5,s_7),}$\\$\mathtt{(s_3,s_6,s_4,s_5),}$\\$\mathtt{(s_1,s_3,s_8,s_5),}$\\$\mathtt{(s_4,s_2,s_6,s_7)}$}
&
{$\mathtt{(a_1,a_2,a_3,s_4)^*,(s_1,a_3,s_3,b_2),(s_6,a_1,s_5,b_1),}$\\
$\mathtt{(a_3,a_4,b_1,s_2),(s_3,b_1,s_4,b_4)^*,(a_4,s_4,b_2,a_1),}$\\
$\mathtt{(b_1,b_2,b_3,s_1)^*,(s_4,b_3,s_2,a_2),(s_6,a_3,s_5,b_3),}$\\
$\mathtt{(b_3,b_4,a_1,s_3),(s_2,a_1,s_1,a_4)^*,(b_4,s_1,a_2,b_1),}$\\
$\mathtt{(a_2,a_4,b_2,b_4),(a_3,b_1,b_3,a_1),(a_2,s_3,a_4,s_6)^*,}$\\
$\mathtt{(b_2,s_2,b_4,s_5)^*,(s_8,b_2,a_3,b_4),(s_7,a_2,b_3,a_4),}$\\
$\mathtt{(s_8,a_1,s_7,b_1),(s_8,a_2,s_5,a_4)^*,(s_8,a_3,s_7,b_3),}$\\
$\mathtt{(s_7,b_2,s_6,b_4)^*}$}
\end{tblr}
\caption{$\widehat{\cD}=\cD\cup(\widehat{\cD}\setminus \cD)$ is a 4-cycle decomposition of $H$, $\cD\subseteq\widehat{\cD}$ is a 4-cycle decomposition of the subgraph $S$ of $H$, and the cycles marked with an asterisk are called the \textit{key cycles}; $\widehat{\cD}$ is anchored to $\{P_1,P_2\}$ since the key cycles cannot be monochromatic (see Lemma \ref{lem: flower form 1} and Lemma \ref{lem: flower form 2}).}
        \label{tbl: flower form}
\end{table}

\begin{proof}
We use the cycle sets $\cD$ and $\widehat{\cD}\setminus \cD$ introduced on the first row of Table~\ref{tbl: flower form} for $(H,S)=(K_9,K_1)$. Note that $\widehat{\cD}=\cD\cup(\widehat{\cD}\setminus \cD)$. Let $\psi:V(H)\rightarrow \{1,2\}$ be a 2-colouring such that $\psi(P_1)=\{1\}$ and $\psi(P_2)=\{2\}$. Since the key cycle $\mathtt{(a_1,a_3,a_2,s_1)}$ cannot be monochromatic, $\psi(\mathtt{s_1})=2$, and so $\widehat{\cD}$ is anchored to $\{P_{1}, P_{2}\}$.
\end{proof}

For all $i\in\{1,\ldots,h\}$, let $\widehat{\cD}_i$ be the 4-cycle decomposition of $H_i$ which we found in Lemma~\ref{lem: flower form 1}, where $\mathtt{s_1}$ is identified with $s_1$, $\mathtt{a_j}$ with $v_{i,2j-1}$, and $\mathtt{b_j}$ with $v_{i,2j}$ for all $j\in\{1,2,3,4\}$. Thus, $\fd_h(\emptyset,\widehat{\cD}_1)$ is a 4-cycle decomposition of $G- G^*$ which, by Lemma~\ref{lem: flower form}, is proto-anchored to $$\left\{\big\{\{v_{1,1},v_{1,3},v_{1,5},v_{1,7}\},\{v_{1,2},v_{1,4},v_{1,6},v_{1,8}\}\big\},\ldots,\big\{\{v_{h,1},v_{h,3},v_{h,5},v_{h,7}\},\{v_{h,2},v_{h,4},v_{h,6},v_{h,8}\}\big\}\right\}.$$

Suppose $h\geq 6$. We define an ordering on each part of $G^*$ with $v_{i,1}<\cdots<v_{i,8}$. Since each part of $G^*$ has exactly eight vertices, $\EA_h$ decomposes $G^*$. Let $\cC=\EA_h\cup\fd_h(\emptyset,\widehat{\cD}_1)$, which is, in fact, a 4-cycle decomposition of $G$. We will prove that for any 2-colouring of $\cC$, $\cC$ is anchored to $\big\{\{v\},\emptyset\big\}$ for some vertex $v\in V(G)$.

We define the 2-colouring $\psi:V(G)\rightarrow\{1,2\}$ where $\psi\big(\bigcup_{i=1}^h\{v_{i,1},v_{i,3},v_{i,5},v_{i,7}\}\big)=\{1\}$ and $\psi\big(\{s_1\}\cup\bigcup_{i=1}^h\{v_{i,2},v_{i,4},v_{i,6},v_{i,8}\}\big)=\{2\}$. The restriction of $\psi$ to $V(G^*)$ is a super-alt-colouring which is a 2-colouring of $\EA_h$. Also, by Lemma~\ref{lem: flower form 1}, $\psi$ admits no monochromatic cycles in $\fd_h(\emptyset,\widehat{\cD}_1)$. Therefore, $\psi$ is a 2-colouring of $\cC$.

For any 2-colouring of $\cC$, $\{v_{1,1}\}$ and $\emptyset$ are heterochromatic. Since $\EA_h$ is exclusively alt-colourable, for all $i$, $\{v_{i,1},v_{i,3},v_{i,5},v_{i,7}\}$ and $\{v_{i,2},v_{i,4},v_{i,6},v_{i,8}\}$ are heterochromatic. Consequently, since $\fd_h(\emptyset,\widehat{\cD}_1)$ is proto-anchored to
$$\big\{\big\{\{v_{1,1},v_{1,3},v_{1,5},v_{1,7}\},\{v_{1,2},v_{1,4},v_{1,6},v_{1,8}\}\big\},\ldots,\big\{\{v_{h,1},v_{h,3},v_{h,5},v_{h,7}\},\{v_{h,2},v_{h,4},v_{h,6},v_{h,8}\}\big\}\big\},$$
it follows by definition that $\bigcup_{i=1}^h\{v_{i,1},v_{i,3},v_{i,5},v_{i,7}\}$
and
$\{s_1\}\cup\bigcup_{i=1}^h\{v_{i,2},v_{i,4},v_{i,6},v_{i,8}\}$
are heterochromatic as well. Hence, $\cC$ is anchored to $\big\{\{v_{1,1}\},\emptyset\big\}$, and so, $\cC$ is uniquely 2-colourable. Thus, we have found a uniquely 2-colourable 4-cycle system of order $8h+1$, proving Theorem~\ref{thm: U24CS} which we now restate.

\UtwofourCS*

Having now established Theorem~\ref{thm: U24CS}, it is time to provide a proof for Theorem~\ref{thm: U24CD}. To do so, we first refine the notation of a cocktail party graph.
Let $I_t$ be a 1-regular graph with $t$ edges, i.e., a 1-factor of order $2t$, and so its complement $\overline{I_t}=K_{2t}- I_t$ is the cocktail party graph of order $2t$. Note that if $I_{t'}$ is a subgraph of $I_t$, then $\overline{I_{t'}}$ is a subgraph of $\overline{I_t}$.

A $4$-cycle decomposition of a cocktail party graph of order $n$ exists only if $n\geq 2$ is even. Let $G=\overline{I_{4h+t}}$ for some positive integers $h$ and $t$, where $$V(G)=\{s_1,\ldots,s_{2t}\}\cup\bigcup_{i=1}^h\{v_{i,1},\ldots,v_{i,8}\}$$ and $$E(I_{4h+t})=\big\{\{s_1,s_2\},\{s_3,s_4\},\ldots,\{s_{2t-1},s_{2t}\}\big\}\cup\bigcup_{i=1}^h\big\{\{v_{i,1},v_{i,2}\},\ldots,\{v_{i,7},v_{i,8}\}\big\}.$$
Let $G^*$ be the complete $h$-partite graph with parts $\{v_{i,1},\ldots,v_{i,8}\}$ (see Figure~\ref{fig: flower form}). Since $I_{4h+t}$ and $G^*$ are edge-disjoint, $G^*$ is a subgraph of $G$.
Now, for all $i\in\{1,\ldots,h\}$, the induced subgraph $H_i$ of $G- G^*$ on $\{s_1,\ldots,s_{2t},v_{i,1},\ldots,v_{i,8}\}$, is a cocktail party graph of order $8+2t$ for which the missing 1-factor has the edge set $$\big\{\{s_1,s_2\},\ldots,\{s_{2t-1},s_{2t}\},\{v_{i,1},v_{i,2}\},\dots,\{v_{i,7},v_{i,8}\}\big\}.$$
Also, let $S$ be the induced subgraph of $G- G^*$ on $\{s_1,\ldots,s_{2t}\}$. Note that $S$ is also a cocktail party graph, in this case of order $2t$. Since all cocktail party graphs of the same order are isomorphic, $H_1,\ldots,H_h$ are isomorphic and they would all be pairwise vertex-disjoint except that $S$ is a common subgraph among all of them. For all $i\neq j$ in $\{1,\ldots,h\}$ we let $\theta_{i,j}:H_i\to H_j$ be the isomorphism where $\theta_{i,j}(v_{i,1})=v_{j,1},\ldots,\theta_{i,j}(v_{i,8})=v_{j,8}$ and $\theta_{i,j}(s_1)=s_1,\ldots,\theta_{i,j}(s_{2t})=s_{2t}$.

\begin{figure}
        \centering
        \begin{tikzpicture}[scale=1.35]

\begin{scope}[scale=0.42,xshift=15cm,yshift=16cm]

\coordinate (C_M) at ( 0,0);

	\coordinate (C_N1) at ( 0:5);
	\coordinate (C_N2) at ( -60:5);
	\coordinate (C_N3) at ( -120:5);
	\coordinate (C_N4) at ( 180:5);
	\coordinate (C_N5) at ( 120:5);
	\coordinate (C_N6) at ( 60:5);
	
	\coordinate (M0)  at ($ (C_M) + ( 0:1)$);
	\coordinate (M1)  at ($ (C_M) + ( -60:1)$);
	\coordinate (M2)  at ($ (C_M) + (-120:1)$);
	\coordinate (M3)  at ($ (C_M) + (180:1)$);
	\coordinate (M4)  at ($ (C_M) + (120:1)$);
	\coordinate (M5)  at ($ (C_M) + (60:1)$);

	\coordinate (N10)  at ($ (C_N1) + ( 90:2)$);
	\coordinate (N11)  at ($ (C_N1) + ( 45:2)$);
	\coordinate (N12)  at ($ (C_N1) + ( 0:2)$);
	\coordinate (N13)  at ($ (C_N1) + (-45:2)$);
	\coordinate (N14)  at ($ (C_N1) + (-90:2)$);
	\coordinate (N15)  at ($ (C_N1) + (-135:2)$);
	\coordinate (N16)  at ($ (C_N1) + (-180:2)$);
	\coordinate (N17)  at ($ (C_N1) + (135:2)$);
	
	\coordinate (N20)  at ($ (C_N2) + ( 90:2)$);
	\coordinate (N21)  at ($ (C_N2) + ( 45:2)$);
	\coordinate (N22)  at ($ (C_N2) + ( 0:2)$);
	\coordinate (N23)  at ($ (C_N2) + (-45:2)$);
	\coordinate (N24)  at ($ (C_N2) + (-90:2)$);
	\coordinate (N25)  at ($ (C_N2) + (-135:2)$);
	\coordinate (N26)  at ($ (C_N2) + (-180:2)$);
	\coordinate (N27)  at ($ (C_N2) + (135:2)$);
	
	\coordinate (N30)  at ($ (C_N3) + ( 90:2)$);
	\coordinate (N31)  at ($ (C_N3) + ( 45:2)$);
	\coordinate (N32)  at ($ (C_N3) + ( 0:2)$);
	\coordinate (N33)  at ($ (C_N3) + (-45:2)$);
	\coordinate (N34)  at ($ (C_N3) + (-90:2)$);
	\coordinate (N35)  at ($ (C_N3) + (-135:2)$);
	\coordinate (N36)  at ($ (C_N3) + (-180:2)$);
	\coordinate (N37)  at ($ (C_N3) + (135:2)$);
	
	\coordinate (N40)  at ($ (C_N4) + ( 90:2)$);
	\coordinate (N41)  at ($ (C_N4) + ( 45:2)$);
	\coordinate (N42)  at ($ (C_N4) + ( 0:2)$);
	\coordinate (N43)  at ($ (C_N4) + (-45:2)$);
	\coordinate (N44)  at ($ (C_N4) + (-90:2)$);
	\coordinate (N45)  at ($ (C_N4) + (-135:2)$);
	\coordinate (N46)  at ($ (C_N4) + (-180:2)$);
	\coordinate (N47)  at ($ (C_N4) + (135:2)$);
	
	\coordinate (N50)  at ($ (C_N5) + ( 90:2)$);
	\coordinate (N51)  at ($ (C_N5) + ( 45:2)$);
	\coordinate (N52)  at ($ (C_N5) + ( 0:2)$);
	\coordinate (N53)  at ($ (C_N5) + (-45:2)$);
	\coordinate (N54)  at ($ (C_N5) + (-90:2)$);
	\coordinate (N55)  at ($ (C_N5) + (-135:2)$);
	\coordinate (N56)  at ($ (C_N5) + (-180:2)$);
	\coordinate (N57)  at ($ (C_N5) + (135:2)$);
	
	\coordinate (N60)  at ($ (C_N6) + ( 90:2)$);
	\coordinate (N61)  at ($ (C_N6) + ( 45:2)$);
	\coordinate (N62)  at ($ (C_N6) + ( 0:2)$);
	\coordinate (N63)  at ($ (C_N6) + (-45:2)$);
	\coordinate (N64)  at ($ (C_N6) + (-90:2)$);
	\coordinate (N65)  at ($ (C_N6) + (-135:2)$);
	\coordinate (N66)  at ($ (C_N6) + (-180:2)$);
	\coordinate (N67)  at ($ (C_N6) + (135:2)$);

	\fill[green!10] (N11) -- (N12) -- (N13) -- (N23) -- (N24) -- (N34) -- (N35) -- (N45) -- (N46) -- (N47) -- (N57) -- (N50) -- (N60) -- (N61) -- cycle;
	
	\foreach \row in {1,...,6} {
        \foreach \col in {0,...,7} {
            \foreach \otherrow in {1,...,6} {
                \ifnum\row<\otherrow
                    \foreach \othercol in {0,...,7} {
                        \draw[green!50!black!90] (N\row\col) -- (N\otherrow\othercol);
                    }
                \fi
            }
        }
    }

        \foreach \row in {1,...,6} {
    \foreach \col in {0,...,7} {
            \fill (N\row\col) circle (3pt);
            }
            }

    \foreach \i in {0,...,5} {
    \fill[green!10] (M\i) circle (3pt);
    \draw[gold(metallic)] (M\i) circle (3pt);
    }

  \node[above=6pt] at (0,-9) {$G^*$};
\end{scope}

\begin{scope}[scale=0.42,xshift=15cm]

	\coordinate (C_M) at ( 0,0);

	\coordinate (C_N1) at ( 0:5);
	\coordinate (C_N2) at ( -60:5);
	\coordinate (C_N3) at ( -120:5);
	\coordinate (C_N4) at ( 180:5);
	\coordinate (C_N5) at ( 120:5);
	\coordinate (C_N6) at ( 60:5);
	
	\coordinate (M0)  at ($ (C_M) + ( 0:1)$);
	\coordinate (M1)  at ($ (C_M) + ( -60:1)$);
	\coordinate (M2)  at ($ (C_M) + (-120:1)$);
	\coordinate (M3)  at ($ (C_M) + (180:1)$);
	\coordinate (M4)  at ($ (C_M) + (120:1)$);
	\coordinate (M5)  at ($ (C_M) + (60:1)$);

	\coordinate (N10)  at ($ (C_N1) + ( 90:2)$);
	\coordinate (N11)  at ($ (C_N1) + ( 45:2)$);
	\coordinate (N12)  at ($ (C_N1) + ( 0:2)$);
	\coordinate (N13)  at ($ (C_N1) + (-45:2)$);
	\coordinate (N14)  at ($ (C_N1) + (-90:2)$);
	\coordinate (N15)  at ($ (C_N1) + (-135:2)$);
	\coordinate (N16)  at ($ (C_N1) + (-180:2)$);
	\coordinate (N17)  at ($ (C_N1) + (135:2)$);
	
	\coordinate (N20)  at ($ (C_N2) + ( 90:2)$);
	\coordinate (N21)  at ($ (C_N2) + ( 45:2)$);
	\coordinate (N22)  at ($ (C_N2) + ( 0:2)$);
	\coordinate (N23)  at ($ (C_N2) + (-45:2)$);
	\coordinate (N24)  at ($ (C_N2) + (-90:2)$);
	\coordinate (N25)  at ($ (C_N2) + (-135:2)$);
	\coordinate (N26)  at ($ (C_N2) + (-180:2)$);
	\coordinate (N27)  at ($ (C_N2) + (135:2)$);
	
	\coordinate (N30)  at ($ (C_N3) + ( 90:2)$);
	\coordinate (N31)  at ($ (C_N3) + ( 45:2)$);
	\coordinate (N32)  at ($ (C_N3) + ( 0:2)$);
	\coordinate (N33)  at ($ (C_N3) + (-45:2)$);
	\coordinate (N34)  at ($ (C_N3) + (-90:2)$);
	\coordinate (N35)  at ($ (C_N3) + (-135:2)$);
	\coordinate (N36)  at ($ (C_N3) + (-180:2)$);
	\coordinate (N37)  at ($ (C_N3) + (135:2)$);
	
	\coordinate (N40)  at ($ (C_N4) + ( 90:2)$);
	\coordinate (N41)  at ($ (C_N4) + ( 45:2)$);
	\coordinate (N42)  at ($ (C_N4) + ( 0:2)$);
	\coordinate (N43)  at ($ (C_N4) + (-45:2)$);
	\coordinate (N44)  at ($ (C_N4) + (-90:2)$);
	\coordinate (N45)  at ($ (C_N4) + (-135:2)$);
	\coordinate (N46)  at ($ (C_N4) + (-180:2)$);
	\coordinate (N47)  at ($ (C_N4) + (135:2)$);
	
	\coordinate (N50)  at ($ (C_N5) + ( 90:2)$);
	\coordinate (N51)  at ($ (C_N5) + ( 45:2)$);
	\coordinate (N52)  at ($ (C_N5) + ( 0:2)$);
	\coordinate (N53)  at ($ (C_N5) + (-45:2)$);
	\coordinate (N54)  at ($ (C_N5) + (-90:2)$);
	\coordinate (N55)  at ($ (C_N5) + (-135:2)$);
	\coordinate (N56)  at ($ (C_N5) + (-180:2)$);
	\coordinate (N57)  at ($ (C_N5) + (135:2)$);
	
	\coordinate (N60)  at ($ (C_N6) + ( 90:2)$);
	\coordinate (N61)  at ($ (C_N6) + ( 45:2)$);
	\coordinate (N62)  at ($ (C_N6) + ( 0:2)$);
	\coordinate (N63)  at ($ (C_N6) + (-45:2)$);
	\coordinate (N64)  at ($ (C_N6) + (-90:2)$);
	\coordinate (N65)  at ($ (C_N6) + (-135:2)$);
	\coordinate (N66)  at ($ (C_N6) + (-180:2)$);
	\coordinate (N67)  at ($ (C_N6) + (135:2)$);

	\fill[red!20] (N21) -- (N22) -- (N23) -- (N24) -- (N25) -- (N26) -- (M3) -- (M5) -- cycle;
	\fill[red!20] (N32) -- (N33) -- (N34) -- (N35) -- (N36) -- (N37) -- (M4) -- (M0) -- cycle;
	\fill[red!20] (N44) -- (N45) -- (N46) -- (N47) -- (N40) -- (M5) -- (M1) -- cycle;
	\fill[red!20] (N55) -- (N56) -- (N57) -- (N50) -- (N51) -- (N52) -- (M0) -- (M2) -- cycle;
	\fill[red!20] (N66) -- (N67) -- (N60) -- (N61) -- (N62) -- (N63) -- (M1) -- (M3) -- cycle;

	\fill[red!20] (N10) -- (N11) -- (N12) -- (N13) -- (N14) -- (M2) -- (M4) -- cycle;

    \foreach \row in {1,...,6} {
    \foreach \col in {0,...,6} {
    \foreach \coll in {\col,...,7} {
    			\ifnum\numexpr\coll-\col\relax=4
      		\else
            \draw[red!50!black] (N\row\col) -- (N\row\coll);
            \fi
            }
            }
            }

    \foreach \i in {0,...,5} {        
    \foreach \row in {1,...,6} {
    \foreach \col in {0,...,7} {
            \draw[thick,red!50!black] (N\row\col) -- (M\i);
            }
            }
            }

    \foreach \row in {1,...,6} {
    \foreach \col in {0,...,7} {
            \fill (N\row\col) circle (3pt);
            }
            }

	\fill[gold(metallic)!15] (M0) -- (M1) -- (M2) -- (M3) -- (M4) -- (M5) -- cycle;

    \foreach \col in {0,...,4} {
    \foreach \coll in {\col,...,5} {
    			\ifnum\numexpr\coll-\col\relax=3
      		\else
            \draw[gold(metallic)] (M\col) -- (M\coll);
            \fi
            }
            }

    \foreach \i in {0,...,5} {
    \fill[gold(metallic)] (M\i) circle (3pt);
    }

  \node[above=6pt] at (0,-9) {$G- G^*$};
\end{scope}

\begin{scope}[scale=0.42,yshift=16cm]

	\coordinate (C_M) at ( 0,0);

	\coordinate (C_N1) at ( 0:5);
	\coordinate (C_N2) at ( -60:5);
	\coordinate (C_N3) at ( -120:5);
	\coordinate (C_N4) at ( 180:5);
	\coordinate (C_N5) at ( 120:5);
	\coordinate (C_N6) at ( 60:5);
	
	\coordinate (M0)  at ($ (C_M) + ( 0:1)$);
	\coordinate (M1)  at ($ (C_M) + ( -60:1)$);
	\coordinate (M2)  at ($ (C_M) + (-120:1)$);
	\coordinate (M3)  at ($ (C_M) + (180:1)$);
	\coordinate (M4)  at ($ (C_M) + (120:1)$);
	\coordinate (M5)  at ($ (C_M) + (60:1)$);

	\coordinate (N10)  at ($ (C_N1) + ( 90:2)$);
	\coordinate (N11)  at ($ (C_N1) + ( 45:2)$);
	\coordinate (N12)  at ($ (C_N1) + ( 0:2)$);
	\coordinate (N13)  at ($ (C_N1) + (-45:2)$);
	\coordinate (N14)  at ($ (C_N1) + (-90:2)$);
	\coordinate (N15)  at ($ (C_N1) + (-135:2)$);
	\coordinate (N16)  at ($ (C_N1) + (-180:2)$);
	\coordinate (N17)  at ($ (C_N1) + (135:2)$);
	
	\coordinate (N20)  at ($ (C_N2) + ( 90:2)$);
	\coordinate (N21)  at ($ (C_N2) + ( 45:2)$);
	\coordinate (N22)  at ($ (C_N2) + ( 0:2)$);
	\coordinate (N23)  at ($ (C_N2) + (-45:2)$);
	\coordinate (N24)  at ($ (C_N2) + (-90:2)$);
	\coordinate (N25)  at ($ (C_N2) + (-135:2)$);
	\coordinate (N26)  at ($ (C_N2) + (-180:2)$);
	\coordinate (N27)  at ($ (C_N2) + (135:2)$);
	
	\coordinate (N30)  at ($ (C_N3) + ( 90:2)$);
	\coordinate (N31)  at ($ (C_N3) + ( 45:2)$);
	\coordinate (N32)  at ($ (C_N3) + ( 0:2)$);
	\coordinate (N33)  at ($ (C_N3) + (-45:2)$);
	\coordinate (N34)  at ($ (C_N3) + (-90:2)$);
	\coordinate (N35)  at ($ (C_N3) + (-135:2)$);
	\coordinate (N36)  at ($ (C_N3) + (-180:2)$);
	\coordinate (N37)  at ($ (C_N3) + (135:2)$);
	
	\coordinate (N40)  at ($ (C_N4) + ( 90:2)$);
	\coordinate (N41)  at ($ (C_N4) + ( 45:2)$);
	\coordinate (N42)  at ($ (C_N4) + ( 0:2)$);
	\coordinate (N43)  at ($ (C_N4) + (-45:2)$);
	\coordinate (N44)  at ($ (C_N4) + (-90:2)$);
	\coordinate (N45)  at ($ (C_N4) + (-135:2)$);
	\coordinate (N46)  at ($ (C_N4) + (-180:2)$);
	\coordinate (N47)  at ($ (C_N4) + (135:2)$);
	
	\coordinate (N50)  at ($ (C_N5) + ( 90:2)$);
	\coordinate (N51)  at ($ (C_N5) + ( 45:2)$);
	\coordinate (N52)  at ($ (C_N5) + ( 0:2)$);
	\coordinate (N53)  at ($ (C_N5) + (-45:2)$);
	\coordinate (N54)  at ($ (C_N5) + (-90:2)$);
	\coordinate (N55)  at ($ (C_N5) + (-135:2)$);
	\coordinate (N56)  at ($ (C_N5) + (-180:2)$);
	\coordinate (N57)  at ($ (C_N5) + (135:2)$);
	
	\coordinate (N60)  at ($ (C_N6) + ( 90:2)$);
	\coordinate (N61)  at ($ (C_N6) + ( 45:2)$);
	\coordinate (N62)  at ($ (C_N6) + ( 0:2)$);
	\coordinate (N63)  at ($ (C_N6) + (-45:2)$);
	\coordinate (N64)  at ($ (C_N6) + (-90:2)$);
	\coordinate (N65)  at ($ (C_N6) + (-135:2)$);
	\coordinate (N66)  at ($ (C_N6) + (-180:2)$);
	\coordinate (N67)  at ($ (C_N6) + (135:2)$);

    \foreach \row in {1,...,6} {
    \foreach \col in {0,...,6} {
    \foreach \coll in {\col,...,7} {
    			\ifnum\numexpr\coll-\col\relax=4
            \draw[red!50!black] (N\row\col) -- (N\row\coll);
            \fi
            }
            }
            }

    \foreach \row in {1,...,6} {
    \foreach \col in {0,...,7} {
            \fill (N\row\col) circle (3pt);
            }
            }

    \foreach \col in {0,...,4} {
    \foreach \coll in {\col,...,5} {
    			\ifnum\numexpr\coll-\col\relax=3
            \draw[gold(metallic)] (M\col) -- (M\coll);
            \fi
            }
            }

    \foreach \i in {0,...,5} {
    \fill[gold(metallic)] (M\i) circle (3pt);
    }

  \node[above=6pt] at (0,-9) {$I_{4h+t}$};
\end{scope}

\begin{scope}[scale=0.42]

\coordinate (C_M) at ( 0,0);

	\coordinate (C_N1) at ( 0:5);
	\coordinate (C_N2) at ( -60:5);
	\coordinate (C_N3) at ( -120:5);
	\coordinate (C_N4) at ( 180:5);
	\coordinate (C_N5) at ( 120:5);
	\coordinate (C_N6) at ( 60:5);
	
	\coordinate (M0)  at ($ (C_M) + ( 0:1)$);
	\coordinate (M1)  at ($ (C_M) + ( -60:1)$);
	\coordinate (M2)  at ($ (C_M) + (-120:1)$);
	\coordinate (M3)  at ($ (C_M) + (180:1)$);
	\coordinate (M4)  at ($ (C_M) + (120:1)$);
	\coordinate (M5)  at ($ (C_M) + (60:1)$);

	\coordinate (N10)  at ($ (C_N1) + ( 90:2)$);
	\coordinate (N11)  at ($ (C_N1) + ( 45:2)$);
	\coordinate (N12)  at ($ (C_N1) + ( 0:2)$);
	\coordinate (N13)  at ($ (C_N1) + (-45:2)$);
	\coordinate (N14)  at ($ (C_N1) + (-90:2)$);
	\coordinate (N15)  at ($ (C_N1) + (-135:2)$);
	\coordinate (N16)  at ($ (C_N1) + (-180:2)$);
	\coordinate (N17)  at ($ (C_N1) + (135:2)$);
	
	\coordinate (N20)  at ($ (C_N2) + ( 90:2)$);
	\coordinate (N21)  at ($ (C_N2) + ( 45:2)$);
	\coordinate (N22)  at ($ (C_N2) + ( 0:2)$);
	\coordinate (N23)  at ($ (C_N2) + (-45:2)$);
	\coordinate (N24)  at ($ (C_N2) + (-90:2)$);
	\coordinate (N25)  at ($ (C_N2) + (-135:2)$);
	\coordinate (N26)  at ($ (C_N2) + (-180:2)$);
	\coordinate (N27)  at ($ (C_N2) + (135:2)$);
	
	\coordinate (N30)  at ($ (C_N3) + ( 90:2)$);
	\coordinate (N31)  at ($ (C_N3) + ( 45:2)$);
	\coordinate (N32)  at ($ (C_N3) + ( 0:2)$);
	\coordinate (N33)  at ($ (C_N3) + (-45:2)$);
	\coordinate (N34)  at ($ (C_N3) + (-90:2)$);
	\coordinate (N35)  at ($ (C_N3) + (-135:2)$);
	\coordinate (N36)  at ($ (C_N3) + (-180:2)$);
	\coordinate (N37)  at ($ (C_N3) + (135:2)$);
	
	\coordinate (N40)  at ($ (C_N4) + ( 90:2)$);
	\coordinate (N41)  at ($ (C_N4) + ( 45:2)$);
	\coordinate (N42)  at ($ (C_N4) + ( 0:2)$);
	\coordinate (N43)  at ($ (C_N4) + (-45:2)$);
	\coordinate (N44)  at ($ (C_N4) + (-90:2)$);
	\coordinate (N45)  at ($ (C_N4) + (-135:2)$);
	\coordinate (N46)  at ($ (C_N4) + (-180:2)$);
	\coordinate (N47)  at ($ (C_N4) + (135:2)$);
	
	\coordinate (N50)  at ($ (C_N5) + ( 90:2)$);
	\coordinate (N51)  at ($ (C_N5) + ( 45:2)$);
	\coordinate (N52)  at ($ (C_N5) + ( 0:2)$);
	\coordinate (N53)  at ($ (C_N5) + (-45:2)$);
	\coordinate (N54)  at ($ (C_N5) + (-90:2)$);
	\coordinate (N55)  at ($ (C_N5) + (-135:2)$);
	\coordinate (N56)  at ($ (C_N5) + (-180:2)$);
	\coordinate (N57)  at ($ (C_N5) + (135:2)$);
	
	\coordinate (N60)  at ($ (C_N6) + ( 90:2)$);
	\coordinate (N61)  at ($ (C_N6) + ( 45:2)$);
	\coordinate (N62)  at ($ (C_N6) + ( 0:2)$);
	\coordinate (N63)  at ($ (C_N6) + (-45:2)$);
	\coordinate (N64)  at ($ (C_N6) + (-90:2)$);
	\coordinate (N65)  at ($ (C_N6) + (-135:2)$);
	\coordinate (N66)  at ($ (C_N6) + (-180:2)$);
	\coordinate (N67)  at ($ (C_N6) + (135:2)$);

	\fill[green!10] (N11) -- (N12) -- (N13) -- (N23) -- (N24) -- (N34) -- (N35) -- (N45) -- (N46) -- (N47) -- (N57) -- (N50) -- (N60) -- (N61) -- cycle;
	
	\foreach \row in {1,...,6} {
        \foreach \col in {0,...,7} {
            \foreach \otherrow in {1,...,6} {
                \ifnum\row<\otherrow
                    \foreach \othercol in {0,...,7} {
                        \draw[green!80!black!20] (N\row\col) -- (N\otherrow\othercol);
                    }
                \fi
            }
        }
    }
    
    \foreach \i in {0,...,5} {
    \fill[green!10] (M\i) circle (3pt);
    \draw[gold(metallic)] (M\i) circle (3pt);
    }
    
    \foreach \row in {6} {
        \foreach \col in {5} {
            \foreach \otherrow in {1,...,5} {
                    \foreach \othercol in {0,...,7} {
                        \draw (N\row\col) -- (N\otherrow\othercol);
                    }
            }
        }
    }

        \foreach \row in {1,...,6} {
    \foreach \col in {0,...,7} {
            \fill (N\row\col) circle (3pt);
            }
            }

    \node[above right=-3pt,align=center] at (N65) {$v_{i,j}$};

  \node[above=6pt,align=center] at (0,-9.5) {$v_{i,j}$ is adjacent to every $v_{i',j'}$ \\ in $G^*$ as long as $i\neq i'$};
\end{scope}

\end{tikzpicture}

        \caption{$I_{4h+t}$ and two prominent subgraphs of $G=\overline{I_{4h+t}}$ for $h=6$ and $t=3$; each of the six groups of vertices around the centre, arranged in a circular fashion, consists of the vertices $v_{i,1},\ldots,v_{i,8}$ for some $i\in\{1,\ldots,6\}$, and the vertices at the centre are the vertices $s_1,\ldots,s_6$. Also shown is the neighbourhood in $G^*$ of a vertex $v_{i,j}$.}
        \label{fig: flower form}
\end{figure}
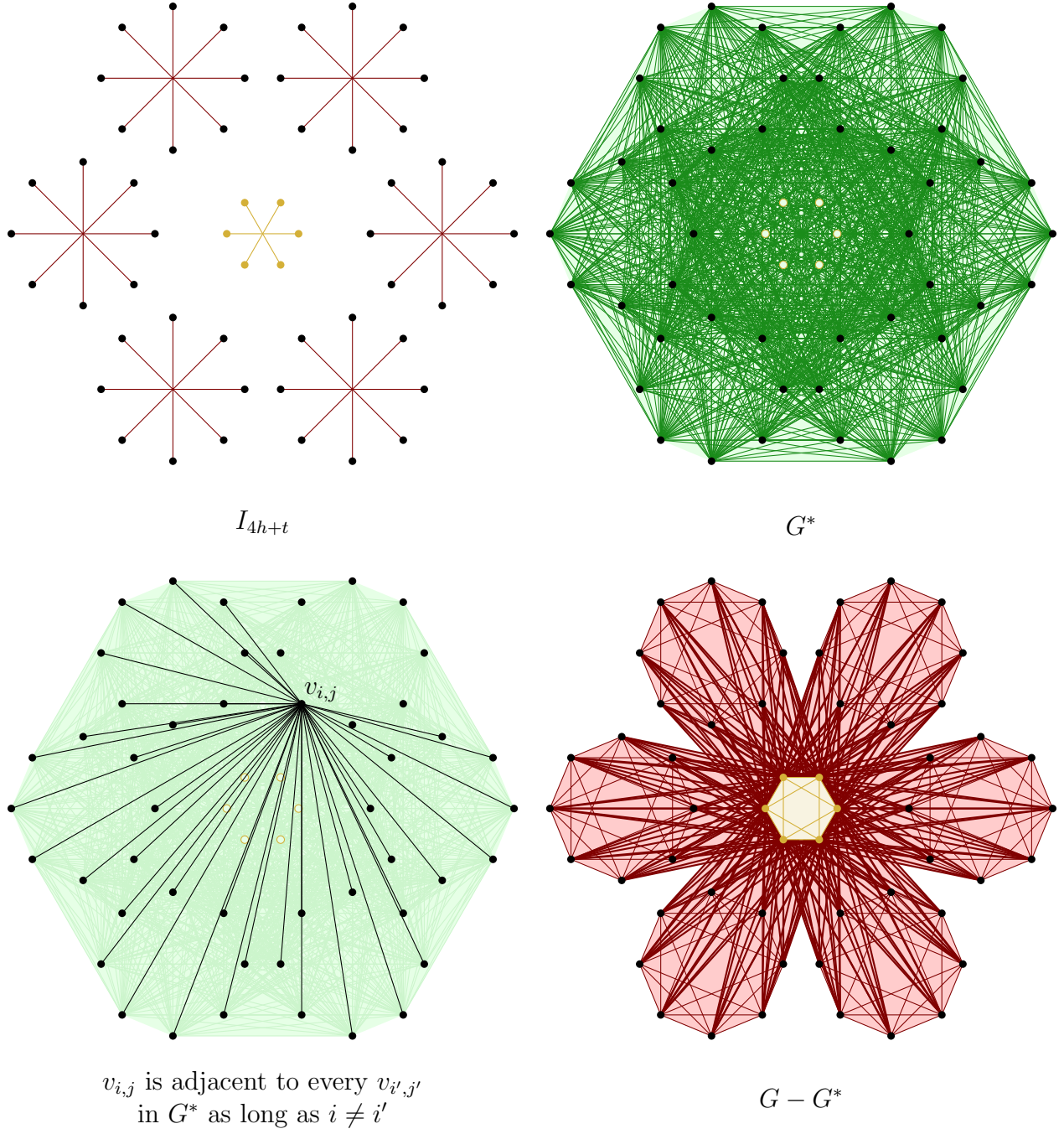

\begin{lemma}

Given a positive integer $t$, let $I_{t+4}$ and $I'_t$ be 1-factors of orders $2t+8$ 
\linebreak
and 
$2t$ respectively such that $E(I'_t)=\big\{\{\mathtt{s_1},\mathtt{s_2}\},\ldots,\{\mathtt{s_{2t-1}},\mathtt{s_{2t}}\}\big\}$ and 
$E(I_{t+4})=E(I'_t)\cup\big\{\{\mathtt{a_1},\mathtt{b_1}\},\ldots,\{\mathtt{a_4},\mathtt{b_4}\}\big\}$. Let $H=\overline{I_{t+4}}$ and $S=\overline{I'_{t}}$ be cocktail party graphs. Then there exist 4-cycle decompositions $\cD$ of $S$ and $\widehat{\cD}$ of $H$, such that $\cD\subseteq \widehat{\cD}$ and $\widehat{\cD}$ is anchored to $\{P_{1}, P_{2}\}$ where $P_{1}=\{\mathtt{a_1,\ldots,a_4}\}$ and $P_{2}=\{\mathtt{b_1,\ldots,b_4}\}$.

\label{lem: flower form 2}
\end{lemma}

\begin{proof}
For clarity, let $(G_t,S_t)=(\overline{I_{t+4}},\overline{I'_{t}})$ for all positive $t$; we want to find $(\widehat{\cD}_t,\cD_t)$ as the corresponding decompositions. Note that for all $t>1$, $E(G_t)=E(S_t)\cup \big(E(G_{t-1})\setminus E(S_{t-1})\big)\cup \big(\bigcup_{i=1}^4 \big\{\{s_{2t-1},a_i\},\{s_{2t},a_i\},\{s_{2t-1},b_i\},\{s_{2t},b_i\}\big\} \big)$.

For $t\in\{1,2,3,4\}$, the cycles of $\widehat{\cD}_t$ and $\cD_t$ are given in Table~\ref{tbl: flower form}, and we construct $\widehat{\cD}_t$ and $\cD_t$ for $t>4$ inductively. For all $t$, let $\varphi_t:V(G_t)\rightarrow V(S_{t+4})$ be the isomorphism from $G_t$ to $S_{t+4}$ such that for all $i$, $\varphi_t(\mathtt{a_{i}})=\mathtt{s_{2i-1}}$, $\varphi_t(\mathtt{b_{i}})=\mathtt{s_{2i}}$, and $\varphi_t(\mathtt{s_{i}})=\mathtt{s_{i+8}}$. For all $t>4$, we define $\cD_t$ as the decomposition that carries over from $\widehat{\cD}_{t-4}$ through the isomorphism $\varphi_{t-4}$, i.e., if $(u_1,u_2,u_3,u_4)\in \widehat{\cD}_{t-4}$, then $\big(\varphi_{t-4}(u_1),\varphi_{t-4}(u_2),\varphi_{t-4}(u_3),\varphi_{t-4}(u_4)\big)\in \cD_t$. For all $t>4$, we let
$$\widehat{\cD}_{t}=\cD_t\cup(\widehat{\cD}_{t-1}\setminus\cD_{t-1})\cup\mathtt{\big\{(s_{2t-1},a_i,s_{2t},b_i)\big\}_{i=1}^4}.$$

Clearly, for all $t$, $\widehat{\cD}_t$ and $\cD_t$ are 4-cycle decompositions of $G_t$ and $S_t$, respectively, and $\cD_t\subseteq \widehat{\cD}_t$. For all $t$, we define the 2-colouring $\psi_t:V(G_t)\rightarrow\{1,2\}$ where $\psi_t(P_1\cup\mathtt{\{s_1,s_3,\ldots,s_{2t-1}\}})=\{1\}$ and $\psi_t(P_2\cup\mathtt{\{s_2,s_4,\ldots,s_{2t}\}})=\{2\}$. It is easy to check that every cycle of $\widehat{\cD}_t$ has at least one vertex in $P_1\cup\mathtt{\{s_1,s_3,\ldots,s_{2t-1}\}}$ and at least one vertex in $P_2\cup\mathtt{\{s_2,s_4,\ldots,s_{2t}\}}$. Therefore, $\psi_t$ is a 2-colouring of $\widehat{\cD}_t$. We will show, by induction, that for all $t$, $\psi_t$ is the only 2-colouring of $\widehat{\cD}_t$ for which $P_1$ and $P_2$ are heterochromatic with colours 1 and 2, respectively and so $\widehat{\cD}_t$ is anchored to $\{P_{1}, P_{2}\}$.

Let $t\in\{1,2,3,4\}$ and for an arbitrary 2-colouring of $\widehat{\cD}_t$ suppose $P_1$ and $P_2$ are heterochromatic with colours 1 and 2, respectively. Since the key cycles (see Table~\ref{tbl: flower form}) cannot be monochromatic, $P_1\cup\mathtt{\{s_1,s_3,\ldots,s_{2t-1}\}}$ and $P_2\cup\mathtt{\{s_2,s_4,\ldots,s_{2t}\}}$ must be heterochromatic as well. Hence, the 2-colouring must be $\psi_t$, and by definition, $\widehat{\cD}_t$ is anchored to $\{P_{1}, P_{2}\}$.

Now let $t>4$ and assume, for all $t'<t$, $\psi_t$ is the only 2-colouring of $\widehat{\cD}_t$ for which $P_1$ and $P_2$ are heterochromatic with colours 1 and 2, respectively. For an arbitrary 2-colouring of $\widehat{\cD}_t$ suppose $P_1$ and $P_2$ are heterochromatic with colours 1 and 2, respectively. For all $t>4$, $\widehat{\cD}_{4}\setminus\cD_{4}\subseteq\widehat{\cD}_{t-1}\setminus\cD_{t-1}\subseteq\widehat{\cD}_t$, and since the key cycles (see Table~\ref{tbl: flower form}) cannot be monochromatic, $P_1\cup\mathtt{\{s_1,s_3,s_5,s_7\}}$ and $P_2\cup\mathtt{\{s_2,s_4,s_6,s_8\}}$ must be heterochromatic. Consequently, $\mathtt{\{s_1,s_3,s_5,s_7\}}$ and $\mathtt{\{s_2,s_4,s_6,s_8\}}$ are heterochromatic. Note that since $\widehat{\cD}_{t-4}$ is anchored to $\{P_{1}, P_{2}\}$, then $\cD_t$ is anchored to $\{\varphi_{t-4}(P_{1}), \varphi_{t-4}(P_{2})\}=\mathtt{\big\{\{s_1,s_3,s_5,s_7\}},\mathtt{\{s_2,s_4,s_6,s_8\}\big\}}$. As a result, $\mathtt{\{s_1,s_3,\ldots,s_{2t-1}\}}$ and $\mathtt{\{s_2,s_4,\ldots,s_{2t}\}}$ must be heterochromatic, and so $P_1\cup\mathtt{\{s_1,s_3,\ldots,s_{2t-1}\}}$ and $P_2\cup\mathtt{\{s_2,s_4,\ldots,s_{2t}\}}$ must be heterochromatic as well. Hence, for all $t$, the 2-colouring must be $\psi_t$, and by definition, $\widehat{\cD}_t$ is anchored to $\{P_{1}, P_{2}\}$.
\end{proof}

Let $h$ and $t$ be positive integers, and for all $i\in\{1,\ldots,h\}$, let $\widehat{\cD}_i$ be the 4-cycle decomposition of $H_i$ which we found in Lemma~\ref{lem: flower form 2} (we denoted that decomposition with $\widehat{\cD}_t$ in the proof), where $\mathtt{s_j'}$ is identified with $s_j'$, $\mathtt{a_j}$ with $v_{i,2j-1}$, and $\mathtt{b_j}$ with $v_{i,2j}$ for all $j\in\{1,2,3,4\}$ and $j'\in\{1,\ldots,2t\}$. Let $\cD$ be the 4-cycle decomposition of $S$ which we also found in Lemma~\ref{lem: flower form 2} (we denoted that decomposition with $\cD_t$ in the proof). Thus, $\fd_h(\cD,\widehat{\cD}_1)$ is a 4-cycle decomposition of $G- G^*$ which, by Lemma~\ref{lem: flower form}, is proto-anchored to $$\big\{\big\{\{v_{1,1},v_{1,3},v_{1,5},v_{1,7}\},\{v_{1,2},v_{1,4},v_{1,6},v_{1,8}\}\big\},\ldots,\big\{\{v_{h,1},v_{h,3},v_{h,5},v_{h,7}\},\{v_{h,2},v_{h,4},v_{h,6},v_{h,8}\}\big\}\big\}.$$

Suppose $h\geq 6$. We define an ordering on each part of $G^*$ with $v_{i,1}<\cdots<v_{i,8}$. Since each part of $G^*$ has exactly eight vertices, $\EA_h$ decomposes $G^*$. Let $\cC=\EA_h\cup\fd_h(\cD,\widehat{\cD}_1)$ be a 4-cycle decomposition of $G$. We will prove that for any 2-colouring of $\cC$, $\cC$ is anchored to $\big\{\{v\},\emptyset\big\}$ for some vertex $v\in V(G)$.

We define the 2-colouring $\psi:V(G)\rightarrow\{1,2\}$ where $$\psi\big(\{s_1,s_3,\dots,s_{2t-1}\}\cup\bigcup_{i=1}^h\{v_{i,1},v_{i,3},v_{i,5},v_{i,7}\}\big)=\{1\}$$ and $$\psi\big(\{s_2,s_4,\dots,s_{2t}\}\cup\bigcup_{i=1}^h\{v_{i,2},v_{i,4},v_{i,6},v_{i,8}\}\big)=\{2\}.$$ The restriction of $\psi$ to $V(G^*)$ is a super-alt-colouring which, by definition, agrees with $\EA_h$. Also, by Lemma~\ref{lem: flower form 2}, $\psi$ admits no monochromatic cycles in $\fd_h(\cD,\widehat{\cD}_1)$. Therefore, $\psi$ is a 2-colouring of $\cC$.

For any 2-colouring of $\cC$, $\{v_{1,1}\}$ and $\emptyset$ are heterochromatic. Since $\EA_h$ is exclusively alt-colourable, for all $i$, $\{v_{i,1},v_{i,3},v_{i,5},v_{i,7}\}$ and $\{v_{i,2},v_{i,4},v_{i,6},v_{i,8}\}$ are heterochromatic. Consequently, since $\fd_h(\cD,\widehat{\cD}_1)$ is proto-anchored to
$$\big\{\big\{\{v_{1,1},v_{1,3},v_{1,5},v_{1,7}\},\{v_{1,2},v_{1,4},v_{1,6},v_{1,8}\}\big\},\ldots,\big\{\{v_{h,1},v_{h,3},v_{h,5},v_{h,7}\},\{v_{h,2},v_{h,4},v_{h,6},v_{h,8}\}\big\}\big\},$$
it follows by definition that
$$\{s_1,s_3,\dots,s_{2t-1}\}\cup\bigcup_{i=1}^h\{v_{i,1},v_{i,3},v_{i,5},v_{i,7}\}$$
and
$$\{s_2,s_4,\dots,s_{2t}\}\cup\bigcup_{i=1}^h\{v_{i,2},v_{i,4},v_{i,6},v_{i,8}\}$$
are heterochromatic as well. Hence, $\cC$ is anchored to $\big\{\{v_{1,1}\},\emptyset\big\}$, and so, $\cC$ is uniquely 2-colourable. Thus, we have found a uniquely 2-colourable 4-cycle decomposition of $\overline{I_{4h+t}}$ for all integers $t\geq 1$ and $h\geq 6$, proving Theorem~\ref{thm: U24CS} which we now restate.

\UtwofourCD*

\section{Discussion}
\label{sec4}

The results of this paper give rise to a natural question: can we use the method we have presented here to obtain results for cycles of other lengths or more colours? With a slight modification and a simple computer search (when feasible) we can find a suitable $\fd_r(\cD,\widehat{\cD}_1)$ to ``uniquify'' an exclusively alt-colourable $m$-cycle decomposition of a complete multipartite graph ($m>3$), in a similar way as we did for the case $m=4$ in Section~\ref{sec3}.
However, the real problem is to find such an exclusively alt-colourable $m$-cycle decomposition. As one might notice, what we achieved in Section~\ref{sec2} relied heavily on Lemma~\ref{lem: a b TPT} which gave us an equation between the number of monochromatic sets of two non-intersecting partitions.
Unfortunately, Lemma~\ref{lem: a b TPT} does not naturally extend to the cases in which the number of colours and the size of the sets in each partition are not 2.
Therefore, we have begun to explore and develop other sophisticated methods of construction to find such decompositions.

In \cite{MR2185517}, Burgess and Pike established that a 3-chromatic 4-cycle system of order $n$ exists for all admissible $n\geq 49$. To this day, 49 is the smallest order for which a 3-chromatic 4-cycle system is known to exist. Curiously, our results prove the existence of uniquely 2-colourable 4-cycle systems of all admissible orders $n\geq 49$. Although we have yet to determine whether 49 is in fact the smallest order for which both problems have a solution, this coincidence cannot be overlooked.
Is there some connection in general between the smallest order of a $k$-chromatic $m$-cycle system and the smallest order of a uniquely $(k-1)$ colourable $m$-cycle system?

\section*{Acknowledgements}
Authors Burgess and Pike acknowledge research grant support from NSERC Discovery Grants RGPIN-2025-04633 and RGPIN-2022-03829, respectively.

\bibliographystyle{plain}
\bibliography{mybibliography1}

\end{document}